\documentclass[runningheads,a4paper]{llncs}

\usepackage{graphicx}
\usepackage{graphics,latexsym,xspace}
\usepackage[all]{xy}
\usepackage{gnuplottex}
\usepackage{lscape}
\usepackage{amssymb}
\usepackage{verbatim}
\usepackage{amsmath}
\usepackage{amsfonts}
\usepackage{tikz}
\newcommand{\ben}{\begin{enumerate}}
\newcommand{\bit}{\begin{itemize}}
\newcommand{\een}{\end{enumerate}}
\newcommand{\eit}{\end{itemize}}

\newcommand{\un}{\underline}
\renewcommand{\a}{\alpha}
\renewcommand{\b}{\beta}

\newcommand{\mc}{\mathcal}

\newcounter{romc}

\newcounter{alphc}
\newcommand{\bd}{\begin{definition}}
\newcommand{\ed}{\end{definition}}
\newcommand{\bp}{\begin{proposition}}
\newcommand{\ep}{\end{proposition}}
\newcommand{\ba}{\begin{eqnarray*}}
\newcommand{\ea}{\end{eqnarray*}}
\newcommand{\bde}{\begin{description}}
\newcommand{\ede}{\end{description}}
\newcommand{\be}{\begin{eqnarray}}
\newcommand{\ee}{\end{eqnarray}}
\newcommand{\bn}{\begin{note}}
\newcommand{\en}{\end{note}}

\newtheorem{notation}{Notation}

\newcommand{\blr}{\begin{list}{~(\roman{romc})~}
{\usecounter{romc}
        \setlength{\topsep}{0pt} \setlength{\itemsep}{0pt}}}
\newcommand{\elr}{\end{list}}
\newcommand{\bla}{\begin{list}{~(\alph{alphc})~} {\usecounter{alphc}
        \setlength{\topsep}{0pt} \setlength{\itemsep}{0pt}}}
\newcommand{\ela}{\end{list}}
\begin{document}

\title{Kleene Algebras and Logic:\\ Boolean and  Rough Set Representations, 3-valued, Rough  and Perp Semantics}

\author{ Arun Kumar and Mohua Banerjee \inst{} }

\institute{
 Department of Mathematics and Statistics, \\Indian Institute of Technology, Kanpur 208016, India\\
\email{arunk2956@gmail.com, mohua@iitk.ac.in}}

\maketitle

\begin{abstract}
A structural theorem for Kleene algebras is proved, showing that an element of a Kleene algebra can be looked upon as an ordered pair of sets. Further, we show that  negation with the Kleene  property (called the `Kleene negation') always arises from the set theoretic complement. The corresponding propositional  logic is then studied through a 3-valued and rough set semantics. It is also established that Kleene negation can be considered as a modal operator, and enables giving a perp semantics to the logic. One concludes with the  observation that all the  semantics for this  logic are equivalent.
\end{abstract}

\section{Introduction}\label{section1}
Algebraists,  since the beginning of work on  lattice theory, have been  keenly interested in representing lattice-based algebras as 
algebras based on {\it set} lattices. Some such well-known  representations are the 
Birkhoff representation for finite lattices,
Stone representation for Boolean algebras, or
Priestley representation for distributive lattices.  It is also well-known that such representation theorems for classes of lattice-based algebras  play a key role in studying set-based semantics of   logics  `corresponding' to the classes. In this paper, we pursue this line of investigation, and focus on {\it Kleene algebras} and their representations. We then move to the corresponding  propositional logic, denoted  $\mathcal{L}_{K}$,  and define a 3-valued, rough set and perp  semantics for it. Through the work here, one is able to establish that  $\mathcal{L}_{K}$  and the 3-element Kleene algebra $\textbf{3}$ (cf. Figure \ref{fig2.1}, Section \ref{section2}) play the same
fundamental role among the Kleene algebras that classical propositional 
logic and the 2-element Boolean algebra $\textbf{2}$   play among the Boolean algebras.


 Kleene algebras were  introduced by Kalman \cite{KAL58} and have been studied under different  names such as normal i-lattices, Kleene lattices and normal quasi-Boolean algebras e.g. cf. \cite{Cignoli65,Cignoli86}. The algebras are defined as follows.
\begin{definition}\label{def1}
An algebra $\mathcal{K} = (K,\vee,\wedge,\sim,0,1)$ is called a {\rm Kleene algebra} if  the following hold.
\begin{enumerate}
\item $\mathcal{K} = (K,\vee,\wedge,\sim,0,1)$ is a De Morgan algebra, i.e., 
      \blr
			\item $(K,\vee,\wedge,0,1)$ is a bounded distributive lattice, and for all $a,b \in K$,
			\item $\sim(a \wedge b) = \sim a ~\vee \sim b$ {\rm (De Morgan property)},
			\item $\sim \sim a = a$ {\rm (involution)}.
			\elr
\item $a~ \wedge \sim a \leq b~ \vee \sim b$, for all $a,b \in K$ {\rm (Kleene property)}.
\end{enumerate}
\end{definition}

  
In order to investigate a representation result for Kleene algebras, it would be natural to first turn to the known representation results for De Morgan algebras, as Kleene algebras are based on them. One  finds the following, in terms of sets.
\bit
\item Rasiowa \cite{Rasiowa1} represented De Morgan algebras as set-based De Morgan algebras, where De Morgan negation  is defined by an involution function.
\item In Dunn's \cite{Dunn66,Dunn99} representation, each element of a De Morgan algebra can be identified with an ordered pair of sets, where De Morgan negation is defined as reversing the order in the pair.  We note that this representation of De Morgan algebras leads to Dunn's famous 4-valued semantics of De Morgan logic.

\eit

On the other hand, we also find that  there are algebras {\it based on} Kleene algebras which can be represented by ordered pairs of sets, and where negations are described 
by set theoretic complements. Consider the set $B^{[2]} := \{(a,b): a\leq b, a,b \in B\}$, for any partially ordered set $(B,\leq)$.
\bit \item  (Moisil (cf. \cite{Cignoli07})) For each 3-valued {\L}ukasiewicz algebra $\mathcal{A}$, there exists a Boolean algebra $B$ such that $\mathcal{A}$ can be embedded into $B^{[2]}$.
     \item  (Katri\v{n}\'{a}k \cite{KATRINAK74}, cf. \cite{Boicescu91}) Every regular double Stone algebra can be embedded into $B^{[2]}$ for some Boolean algebra $B$.
\eit
Rough set theory \cite{pawlak82,pawlak91} also provides a way to represent algebras as pairs of sets. In rough set terminology (that will be elaborated on in Section \ref{section4}), we have the following.
\bit \item (Comer \cite{Comer95}) Every regular double Stone algebra is isomorphic to an algebra of rough sets in a Pawlak approximation space. 

\item  (J\"{a}rvinen \cite{JR11}) 
Every Nelson algebra defined over an algebraic lattice is isomorphic to an algebra of rough sets in an approximation space based on a quasi order. 
\eit
 
In this article, the following representation results are established for Kleene algebras.
\begin{theorem} \label{mainthm} \noindent 
\blr \item Given a Kleene algebra $\mathcal{K}$, there exists a Boolean algebra $\mathcal{B_{K}}$ such that 
$\mathcal{K}$ can be embedded into $\mathcal{B_{K}}^{[2]}$. 
\item Equivalently, every  Kleene algebra  is isomorphic to an algebra of rough sets in a Pawlak approximation space. 
\elr
\end{theorem} 

The De Morgan negation operator with the Kleene property mentioned in Definition \ref{def1}, is referred to as the {\it Kleene negation}. In literature, one finds various generalizations of the classical (Boolean) negation, including the De Morgan and Kleene negations. 
It is natural to ask the following question:
do these generalized negations arise from (or can be described by) the Boolean negation? The representation result above (Theorem \ref{mainthm}) for Kleene algebras shows that 
Kleene algebras always arise from Boolean algebras, thus answering the above question in the affirmative for the Kleene negation.


We next proceed to study  the  logic $\mathcal{L}_{K}$ corresponding to the class of Kleene algebras.  $\mathcal{L}_{K}$ is the De Morgan consequence system \cite{Dunn99} with the negation operator satisfying the {\it Kleene axiom}: $\alpha~ \wedge \sim \alpha \vdash \beta~ \vee \sim \beta$. We show that $\mathcal{L}_{K}$ is  sound and complete with respect to a 3-valued as well as a  rough set semantics. Moroever, it can be imparted a {\it perp semantics} \cite{Dunn05}. 
In fact, the perp semantics  provides a framework to study various negations from the logic as well as algebraic points of view. In particular, 
De Morgan logic is sound and complete with respect to a class of perp ({\it compatibility}) frames, and thus the algebraic semantics, 4-valued semantics and perp semantics for De Morgan logic coincide. In  case of the  logic  $\mathcal{L}_{K}$, we obtain that  its algebraic,   3-valued,  rough set  and perp semantics  are all equivalent. 

The paper is organized as follows. In Section \ref{section2}, we prove (i) of Theorem \ref{mainthm}, in the form of Theorem \ref{thm2.1}. The logic $\mathcal{L}_{K}$ and its 3-valued semantics are introduced in Section \ref{section3}, further we prove soundness and completeness results. In Section \ref{section4}, we establish a rough set representation of Kleene algebras, that is, (ii) of Theorem  \ref{mainthm}, relate rough sets with the 3-valued semantics considered in this work, and finally present completeness of $\mathcal{L}_{K}$ with respect to the rough set  semantics. 
In Section \ref{section5}, we discuss the perp semantics for $\mathcal{L}_{K}$, and  investigate the Kleene property in perp frames. Section \ref{section5} ends with the observation that  all the semantics defined for $\mathcal{L}_{K}$ are equivalent (Theorem \ref{thm20}). We conclude the article in Section \ref{section6}. 

The lattice theoretic results used in this article are taken from \cite{davey}. We  use the convention of representing a set $\{x,y,z,...\}$ by $xyz....$.

\section{Boolean representation of Kleene algebras} \label{section2}

Construction of new types of algebras from a given algebra has been of  prime interest for algebraists, especially  in the context of algebraic logic.  Some well known examples of such construction are:
\bit
\item  Nelson algebra from a given Heyting algebra (Vakarelov \cite{Vakarelov77,Fidel78}).
\item Kleene algebras from distributive lattices (Kalman \cite{KAL58}).
\item 3-valued {\L}ukasiewicz-Moisil (LM) algebra from a given Boolean algebra (Moisil, cf. \cite{Cignoli07}).
\item Regular double Stone algebra from a Boolean algebra  (Katri\v{n}\'{a}k \cite{KATRINAK74}, cf. \cite{Boicescu91}).
\eit
\noindent Our work is based on  the Moisil construction of a 3-valued LM algebra (which is, in particular, a Kleene algebra). Let us present this construction.\vskip 2pt 
\noindent Let $\mathcal{B} := (B,\vee,\wedge,^{c},0,1)$ be a Boolean algebra. Consider again, the set \vskip 2pt 
\hspace*{2 cm} $B^{[2]} := \{(a,b): a\leq b, a,b \in B\}$. 

\begin{proposition}\label{prop1}
$\mathcal{B}^{[2]} := (B^{[2]}, \vee, \wedge, \sim, (0,0), (1,1))$ is a Kleene algebra, where,\\
\hspace*{2 cm} $(a,b) \vee (c,d) := (a\vee c,b \vee d )$,\\
\hspace*{2 cm} $(a,b) \wedge (c,d) := (a\wedge c,b \wedge d )$,\\ 
\hspace*{2 cm} $\sim(a,b) := (b^{c},a^{c})$.
\end{proposition}
\begin{proof} Let us  only prove the Kleene property for $\sim$. Let $(a,b), (c,d) \in B^{[2]}$. 

\noindent $(a,b) \wedge \sim (a,b) = (a,b) \wedge (b^{c},a^{c}) = (a \wedge b^{c},b \wedge a^{c}) = (0,b \wedge a^{c})$.

\noindent $(c,d) \vee \sim (c,d) = (c,d) \vee (d^{c},c^{c}) = (c \vee d^{c}, d \vee c^{c}) = (c \vee d^{c},1)$.

\noindent Hence $(a,b) \wedge \sim (a,b) \leq (c,d) \vee \sim (c,d)$.

\end{proof}

\noindent In this section we  prove  the following representation result. 

\begin{theorem} \label{thm2.1}
Given a Kleene algebra $\mathcal{K}$, there exists a Boolean algebra $\mathcal{B_{K}}$ such that 
$\mathcal{K}$ can be embedded into $\mathcal{B_{K}}^{[2]}$.
\end{theorem}

\noindent Observe that we already have the following well-known representation theorem, due to the fact that $\textbf{1,~2}$ and $\textbf{3}$ (Figure \ref{fig2.1}) are the only subdirectly irreducible (Kleene) algebras
in the variety of Kleene algebras.

\begin{theorem} \label{thm2.2} \cite{Balbes74}
Let $\mathcal{K}$ be a Kleene algebra. There exists a (index) set $I$ such that $\mathcal{K}$ can be 
embedded into {\rm$\textbf{3}^{I}$}.

\begin{figure}[h] 
\begin{tikzpicture}[scale=.75]
    
    \draw [dotted] (0,1) -- (0,1);

		\draw [fill] (3,2) circle [radius = 0.1];
		\node [right] at (3.25,2) {$a = \sim a$};
		\node [right] at (1.5,2) {$\textbf{1}$ := };

		\node [right] at (6,2) {$\textbf{2}$ := };
		\draw (8,1) -- (8,3);
		\draw [fill] (8,1) circle [radius = 0.1];
		\node [below] at (8,1) {0 = $\sim$ 1};
		\draw [fill] (8,3) circle [radius = 0.1];
		\node [above] at (8,3) {1 = $\sim$ 0};
		
		\node [right] at (11,2) {$\textbf{3}$ := };
		\draw (13,0) -- (13,2);
		\node [below] at (13,0) {0 = $\sim$ 1};
    \draw [fill] (13,0) circle [radius = 0.1];
		\node [right] at (13,2) {$a = \sim a$};
    \draw [fill] (13,2) circle [radius = 0.1];
		\draw (13,2) -- (13,4);
		\draw [fill] (13,4) circle [radius = 0.1];
		\node [above] at (13,4) {1 = $\sim$ 0};

\end{tikzpicture}
\caption{Subdirectly irreducible Kleene algebras}

\label{fig2.1}
\end{figure}
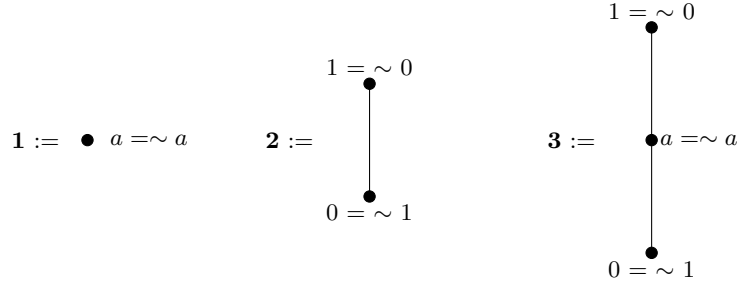
\end{theorem} 
So, to prove the Theorem \ref{thm2.1}, we prove the following.
\begin{theorem} \label{thm2.3}
For the Kleene algebra {\rm$\textbf{3}^{I}$} correponding to any index set $I$, there exists 
a Boolean algebra {\rm$B_{\textbf{3}^{I}}$} such that 
{\rm$\textbf{3}^{I} \cong \mathcal{(B_{\textbf{3}^{I}})}^{[2]}$}.

\end{theorem}

\subsection{Completely join irreducible elements of $\textbf{3}^{I}$ and $(\textbf{2}^{I})^{[2]}$}\label{cji}

Completely join irreducible elements play a fundamental role in establishing  isomorphisms between  lattice-based algebras. Let us put the basic definitions and notations in place.
\begin{definition}\label{cjijd} Let  $\mathcal{L}: =
(L,\vee,\wedge,0,1)$ be a complete lattice. \blr \item An element $a
\in L$ is said to be {\rm completely join irreducible}, when $a = \vee
a_{i}$ implies that $a = a_{i}$ for some $i$. 
 \begin{notation}\label{not1} Let
$\mathcal{J}_{L}$ denote the set of all   completely
join irreducible elements of $L$, and  $J(x) := \{a \in \mathcal{J}_{L}: a \leq x\}$, for any $x \in L$. \end{notation} 

\item A set $S$ is said to be {\rm join
dense} in $\mc L$, provided every element of $L$ is the join of some
elements from $S$. \elr
\end{definition}
\noindent Observe that the element $0 $ is always completely join irreducible.

 For a given index set $I$, let us characterize the sets of completely join irreducible elements of the Kleene algebras 
$\textbf{3}^{I}$ and $(\textbf{2}^{I})^{[2]}$, and prove their join density in the respective lattices.  
\noindent Let $i,k \in I$. Denote by  $f_{i}^{x},~x \in \textbf{3}:=\{0,a,1\},$ 
the following element in $\textbf{3}^{I}$. \vskip 2pt
\hspace*{2 cm}$f_{i}^{x}(k)  :=
\left\{
	\begin{array}{ll}
		x  & \mbox{if } k = i \\
		0 & otherwise
	\end{array}
\right.$  

%

\begin{proposition}
The set of completely join irreducible  elements of {\rm $\textbf{3}^{I}$}  is given by:\\
\hspace*{2 cm}{\rm$\mathcal{J}_{\textbf{3}^{I}} = \{f_{i}^{x}: i \in I, x \in  \textbf{3} \}$}.\vskip 2pt
\noindent Moreover, {\rm$\mathcal{J}_{\textbf{3}^{I}}$} 
is join dense in {\rm$\textbf{3}^{I}$}.
\end{proposition}
\begin{proof}
We only consider the non-zero elements. Let $f_{i}^{a} = \vee_{k \in K}f_{k},~K \subseteq I$. This implies that $f_{i}^{a}(j) = \vee_{k \in K}f_{k}(j)$, for each $j \in I$.
If $j\neq i$, by the definition of $f_{i}^{a}$,  $f_{i}^{a}(j) = 0$. 
So $\vee_{k \in K}f_{k}(j) = 0$, whence $f_{k}(j) = 0,$ for each $k \in K$. If $j = i$, then $f_{i}^{a}(j) = a$, which means $\vee_{k \in K}f_{k}(j) = a$. But as $a$ is join irreducible in $\textbf{3}$,  there exists a $k^{\prime} \in K$ such that $f_{k^{\prime}}(j) = a$. Hence $f_{i}^{a} = f_{k^{\prime}}$. A similar argument works for $f_{i}^{1}$. \\
\noindent Now let $f (\neq 0) \in \textbf{3}^{I}$. Take $K  = I$, and for each $j \in I$, define the element $f_{j}$  of $\textbf{3}^{I}$ as\\
\hspace*{2 cm}$f_{j}(k)  :=
\left\{
	\begin{array}{ll}
		f(j)  & \mbox{if } k = j \\
		0 & otherwise
	\end{array}
\right.$ \\
Clearly, we have $f = \vee_{j \in I}f_{j}$, where $f_j \in \mathcal{J}_{\textbf{3}^{I}}$.
\qed
\end{proof}
\noindent Let us note that for each $i,j \in I$, $f_{i}^{a} \leq f_{i}^{1}$, and if $i \neq j$,
neither $f_{i}^{x} \leq f_{j}^{y}$ nor $f_{j}^{x} \leq f_{i}^{y}$ holds for $x,y \in \{a,1\}$.
The order structure of the non-zero elements in $\mathcal{J}_{\textbf{3}^{I}}$ can be visualized by Figure \ref{fig1}:  
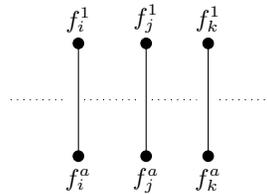
\begin{figure}[h] \label{fig1}
\begin{tikzpicture}[scale=.75]
    \tikzstyle{every node}=[draw,circle,fill=black,minimum size=4pt,
                            inner sep=0pt]

    \draw [dotted] (0,1) -- (0,1);
    \draw [dotted] (6,1) -- (7,1);
    
    \draw (7.2,0) -- (7.2,2);
		\draw (7.2,0) node [label=below:{\it $f_{i}^{a}$}]{};
		\draw (7.2,2) node [label=above:{\it $f_{i}^{1}$}]{};
		\draw [dotted] (7.3,1) -- (8.3,1);
		\draw (8.4,0) -- (8.4,2);
		\draw (8.4,0) node [label=below:{\it $f_{j}^{a}$}]{};
		\draw (8.4,2) node [label=above:{\it $f_{j}^{1}$}]{};
		\draw [dotted] (8.5,1) -- (9.3,1);
		\draw (9.5,0) -- (9.5,2);
		\draw (9.5,0) node [label=below:{\it $f_{k}^{a}$}]{};
		\draw (9.5,2) node [label=above:{\it $f_{k}^{1}$}]{};
		\draw [dotted] (9.7,1) -- (10.7,1);

\end{tikzpicture}
\caption{Hasse diagram of $\mathcal{J}_{\textbf{3}^{I}}$}
\label{fig1}
\end{figure}

\begin{example}
Let us consider the Kleene algebra $\textbf{3}^{3}$. The set $\mathcal{J}_{\textbf{3}^{3}}$ of completely join irreducible  elements of $\textbf{3}^{3}$ is then given by\\
\noindent $\mathcal{J}_{\textbf{3}^{3}} = \{ (0,0,0), f_{1}^{a} := (a,0,0), f_{1}^{1} := (1,0,0), f_{2}^{a} := 
(0,a,0), f_{2}^{1} := (0,1,0), f_{3}^{a} := (0,0,a), f_{3}^{1} := (0,0,1)\}$.

\noindent Let $f := (0,a,1) \in \textbf{3}^{3}$. Then $f =  f_{1} \vee f_{2} \vee f_{3}$, where  $f_{1} = (0,0,0),~f_{2} = (0,a,0)$
and $f_{3} = (0,0,1)$. 
\end{example} 

\noindent As any complete atomic Boolean algebra is isomorphic to $\textbf{2}^{I}$ for some index set $I$, henceforth, we shall identify any complete atomic Boolean algebra $B$ with $\textbf{2}^{I}$. 
Now, for any such algebra, $B^{[2]}$ is a Kleene algebra (cf. Proposition \ref{prop1}); in fact, it is a completely
distributive Kleene algebra.

\begin{proposition}\label{prop3}
Let $B$ be a complete atomic Boolean algebra. The set of completely join irreducible  elements  of $B^{[2]}$ is given by\\
\hspace*{2 cm} $\mathcal{J}_{B^{[2]}} = \{(0,a),(a,a): a \in \mathcal{J}_{B}\}$.\vskip 2pt 
\noindent Moreover, $\mathcal{J}_{B^{[2]}}$ is join dense in $B^{[2]}$. 
\end{proposition}
\begin{proof}
Let $a \in \mathcal{J}_{B}$ and let $(a,a) = \vee_{k \in K}(x_{k},y_{k}),~K \subseteq I$, 
where $(x_{k},y_{k}) \in B^{[2]}$ for each $k \in K$. $(a,a) = \vee_{k \in K}(x_{k},y_{k})$ implies $a = \vee_{k}x_{k}$. As
$a \in \mathcal{J}_{B}$,  $a = x_{k^{\prime}}$ for some $k^{\prime} \in K$. 
We already have $x_{k^{\prime}} \leq y_{k^{\prime}} \leq a$, hence combining with $a = x_{k^{\prime}}$, we get $(a,a) = (x_{k^{\prime}},y_{k^{\prime}})$.
With similar arguments one can show that for each $a \in \mathcal{J}_{B}$, $(0,a)$ is completely join irreducible.

Now, let $(x,y) (\neq (0,0)) \in B^{[2]}$. Consider the sets $J(x),~J(y)$ (cf. Notation \ref{not1}, Definition \ref{cjijd}). 
 Then $(x,y) = \vee_{a \in J(x)}(a,a)~ \vee~ \vee_{b \in J(y)}(0,b)$. Hence $\mathcal{J}_{B^{[2]}}$ is join dense in $B^{[2]}$.\qed
\end{proof}

For $a, b ~(\neq 0) \in \mathcal{J}_{B}$, $(0,a) \leq (a,a)$, and if $a \neq b, ~x,y \in \{a,b\}$ with $x \neq y$, neither $(0,x) \leq (0,y),(y,y)$ nor $(x,x) \leq (0,y),(y,y)$
holds. 
Then, similar to the case of $\textbf{3}^{I}$, the completely join irreducible non-zero elements of $B^{[2]}$ can be visualized by Figure \ref{fig2}:
\begin{figure}[h] 
\begin{tikzpicture}[scale=.75]
    \tikzstyle{every node}=[draw,circle,fill=black,minimum size=4pt,
                            inner sep=0pt]

    \draw [dotted] (0,1) -- (0,1);
    \draw [dotted] (6,1) -- (7,1);
  
    \draw (7.5,0) -- (7.5,2);
		\draw (7.5,0) node [label=below:{\it $(0,a_{i})$}]{};
		\draw (7.5,2) node [label=above:{\it $(a_{i},a_{i})$}]{};
		\draw [dotted] (7.8,1) -- (8.5,1);
		\draw (8.8,0) -- (8.8,2);
		\draw (8.8,0) node [label=below:{\it $(0,a_{j})$}]{};
		\draw (8.8,2) node [label=above:{\it $(a_{j},a_{j})$}]{};
		\draw [dotted] (9,1) -- (10,1);
		\draw (10.2,0) -- (10.2,2);
		\draw (10.2,0) node [label=below:{\it $(0,a_{k})$}]{};
		\draw (10.2,2) node [label=above:{\it $(a_{k},a_{k})$}]{};
		\draw [dotted] (10.4,1) -- (11,1);
		
\end{tikzpicture}
\caption{Hasse diagram of $\mathcal{J}_{B^{[2]}}$}
\label{fig2}
\end{figure}
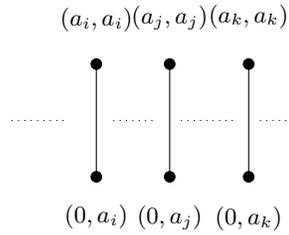

\begin{example}
Consider the Boolean algebra $\textbf{4}$  of four elements with atoms $a$ and $b$. The set of completely join irreducible  elements of $\textbf{4}^{[2]}$ is given by\\
\noindent $\mathcal{J}_{\textbf{4}^{[2]}} = \{(0,0),(0,a),(a,a),(0,b),(b,b)\}$. 

\noindent Let $(a,1) \in \textbf{4}^{[2]}$. Then   $J(a) = \{0,a\}$ and $J(1) = \{0,a,b\}$. Hence\\ $(a,1) = (0,0) \vee (a,a) \vee (0,a) \vee (0,b)$. 

\end{example}

\subsection{Structural theorem for Kleene algebras}

Let us first present the basic lattice-theoretic  definitions and results that will be required to arrive at the proof of Theorem \ref{thm2.3}.
\begin{definition} \label{Def3}\noindent 
\begin{enumerate}
\item 
\begin{enumerate}
\item A {\rm complete lattice of sets} is a family $\mathcal{F}$ such that $\bigcup \mathcal{H}$ and $\bigcap \mathcal{H}$
belong to $\mathcal{F}$ for any $\mathcal{H} \subseteq \mathcal{F}$.
\end{enumerate}
\item Let $L$ be a complete lattice.
\begin{enumerate}
\item $L$ is said to be {\rm algebraic} if any element $x \in L$
is the join of a set of compact elements of $L$.

\item $L$ is said to  satisfy the {\rm Join-Infinite Distributive Law}, if  for any subset $\{y_{j}\}_{j \in J}$ of $L$ and any $x \in L$,\\
$$ (JID)~~ x \wedge \bigvee_{j \in J} y_{j} = \bigvee_{j \in J} x \wedge y_{j}.$$ 
\end{enumerate}
\end{enumerate}
\end{definition}  

\begin{theorem} \label{thm2.7} \cite{davey}
Let $L$ be a lattice. The following are equivalent.
\begin{enumerate}

\item $L$ is complete,  satisfies $(JID)$ and the set of completely join irreducible elements is join dense in L.

\item $L$ is completely distributive and  algebraic.

\end{enumerate}
\end{theorem}
\noindent It can be easily seen that both the lattices 
$\textbf{3}^{I}$ and $(\textbf{2}^{I})^{[2]}$ are complete and satisfy $(JID)$. We have already observed from  Section \ref{cji} that the sets of completely join irreducible elements of $\textbf{3}^{I}$ and $(\textbf{2}^{I})^{[2]}$ are join dense in the respective lattices. So   Theorem \ref{thm2.7}(1) holds for $\textbf{3}^{I}$ and $(\textbf{2}^{I})^{[2]}$, and  
therefore, $\textbf{3}^{I}$ and $(\textbf{2}^{I})^{[2]}$ are completely distributive and algebraic lattices.\vskip 2pt

\begin{note} For the remaining study, let us fix an index set $I$. In the rest of our paper, we exclude $0$  from  the list of completely join irreducible elements and the lattice in the definition of join density, as this does not change the results.
\end{note} 

Let us recall from  Section \ref{cji} that  the completely join irreducible elements of $\textbf{3}^{I}$ and $(\textbf{2}^{I})^{[2]}$ are given by {\rm $\mathcal{J}_{\textbf{3}^{I}} = \{f_{i}^{a},f_{i}^{1}: i\in I\}$ and 
$\mathcal{J}_{(\textbf{2}^{I})^{[2]}} = \{(0,g_{i}^{1}),(g_{i}^{1},g_{i}^{1}):i \in I\}$}, 
where $g_{i}^{1}$'s are the atoms or non-zero completely join irreducible elements of the Boolean algebra 
$\textbf{2}^{I}$,  defined as the earlier $f_{i}^{1}$, with domain restricted to $\textbf{2}$. \vskip 2pt
\hspace*{2 cm}$g_{i}^{1}(k)  :=
\left\{
	\begin{array}{ll}
		1  & \mbox{if } k = i \\
		0 & otherwise
	\end{array}
\right.$

\begin{theorem}\label{join irr}
The  sets of completely join irreducible elements of 
 {\rm $\textbf{3}^{I}$} and {\rm $(\textbf{2}^{I})^{[2]}$} are order isomorphic.
\end{theorem}
\begin{proof}
We define the map 
{\rm $\phi: \mathcal{J}_{\textbf{3}^{I}} \rightarrow \mathcal{J}_{(\textbf{2}^{I})^{[2]}}$} as follows. For $i \in I$,\\
\hspace*{2 cm} $\phi(f_{i}^{a}) := (0,g_{i}^{1})$,\\
\hspace*{2 cm} $\phi(f_{i}^{1}) := (g_{i}^{1},g_{i}^{1})$.\vskip 2pt 
\noindent One can show that $\phi$ is an order isomorphism.
\bit
\item $f_{i}^{x} \leq f_{j}^{y}$ if and only if $i = j$ and $x,y = a$ or $x,y = 1$, or  $x = a, y = 1$. In any case, by definition of 
$\phi$, $\phi(f_{i}^{x}) \leq \phi(f_{j}^{y})$.

\item Let $\phi(f_{i}^{x}) \leq \phi(f_{j}^{y})$ and assume $\phi(f_{i}^{x}) = (g_{k}^{1},g_{l}^{1})$ and 
$\phi(f_{j}^{y}) = (g_{m}^{1},g_{n}^{1})$. But then again:  $k = l = m = n$ or $g_{k}^{1} = g_{m}^{1} = 0, l = n $ or $g_{k}^{1} = 0, l = m = n$.
Again, following the definition of $\phi$, we have $f_{i}^{x} \leq f_{j}^{y}$.

\item If $(0,g_{i}^{1}) \in \mathcal{J}_{(\textbf{2}^{I})^{[2]}}$, then $\phi(f_{i}^{a}) = (0,g_{i}^{1})$. Similarly 
for $(g_{i}^{1},g_{i}^{1})$. Hence $\phi$ is onto.\qed
\eit 
\end{proof}

\begin{lemma} \cite{BIRK95} \label{lemma2.1}
Let $L$ and $K$ be two completely distributive lattices. Further, let $\mathcal{J}_{L}$ and $\mathcal{J}_{K}$ be join dense in $L$ and
$K$ respectively. Let $\phi:\mathcal{J}_{L} \rightarrow \mathcal{J}_{K}$ be an order isomorphism. Then the extension map $\Phi: L \rightarrow K$ given by \\
\hspace*{2 cm} $\Phi(x) := \bigvee(\phi(J(x)))$ (where $J(x) := \{a \in \mathcal{J}_{L}: a \leq x\}$), $x \in L$,\\
is a lattice isomorphism.
\end{lemma}
Using  Theorem \ref{join irr} and Lemma \ref{lemma2.1}, we have the following:
\begin{theorem} \label{thm2.8}
 The  algebras {\rm $\textbf{3}^{I}$} and {\rm $(\textbf{2}^{I})^{[2]}$}
are lattice isomorphic.
\end{theorem}
\noindent Now, in order to obtain Theorem \ref{thm2.3},  we would like to extend the above lattice isomorphism to a  Kleene isomorphism. 
We use the technique of J\"{a}rvinen in \cite{JR11}. Let us present the preliminaries. \\
Let $\mathbb{K} := (K,\vee,\wedge,\sim,0,1)$ be a completely distributive De Morgan algebra. Define for any $j \in \mathcal{J}_{K}$,\\
\hspace*{3 cm} $j^{*} := \bigwedge \{x \in K: x \nleq  \sim j\}$.\\
Then $j^{*} \in \mathcal{J}_{K}$. For complete details on $j^{*}$, one may refer to \cite{JR11}. Further, it is shown that 
Lemma \ref{lemma2.1} can be extended to  De Morgan algebras defined over  algebraic lattices. 

\begin{theorem} \label{thm2.10}
Let $\mathbb{L} := (L,\vee,\wedge,\sim,0,1)$ and $\mathbb{K} := (K,\vee,\wedge,\sim,0,1)$ be two De Morgan algebras defined on algebraic lattices. If $\phi: \mathcal{J}_{L} \rightarrow \mathcal{J}_{K}$ is an order isomorphism such that \\
\hspace*{3 cm} $\phi(j^{*}) = \phi(j)^{*},$ for all $j \in \mathcal{J}_{L}$,\\
 then $\Phi$ is an isomorphism between the algebras $\mathbb{L}$ and $\mathbb{K}$.

\end{theorem} 
Let  $f_{i}^{a} \in \mathcal{J}_{\textbf{3}^{I}}$. Then by definition, 
$(f_{i}^{a})^{*} = \bigwedge \{f \in \textbf{3}^{I}: f \nleq \sim(f_{i}^{a})\}$, where for each $i \in I$, \\
\hspace*{2 cm}$\sim(f_{i}^{a})(k)  =
\left\{
	\begin{array}{ll}
		a  & \mbox{if } k = i \\
		1 & otherwise
	\end{array}
\right.$\\
Clearly, we have $f_{i}^{1} \nleq \sim(f_{i}^{a})$. Now let $f \nleq \sim(f_{i}^{a})$. Then what does $f$  look like?
If $k \neq i$, $f(k) \leq \sim(f_{i}^{a})(k) = 1$. So, for $f \nleq \sim(f_{i}^{a})$, $f(i)$ has to be $1$ (otherwise 
$f(i) = 0,a$ will lead to $f \leq \sim(f_{i}^{a})$). Hence,  $f_{i}^{1} \leq f$ 
and $(f_{i}^{a})^{*} = f_{i}^{1}$.\\
Similarly, one can easily show that $(f_{i}^{1})^{*} = f_{i}^{a}$.\vskip 2pt 
\noindent On the other hand, let us consider $(0,g_{i}^{1}) \in \mathcal{J}_{(\textbf{2}^{I})^{[2]}}$. Then, 
$(0,g_{i}^{1})^{*} = \bigwedge \{(g,g^{\prime}) \in (\textbf{2}^{I})^{[2]}: (g,g^{\prime}) \nleq \sim(0,g_{i}^{1})\}$. By definition of $\sim$, we have $\sim(0,g_{i}^{1}) = ((g_{i}^{1})^{c},0^{c}) = ((g_{i}^{1})^{c},1)$. 
 Observe that $(g_{i}^{1},g_{i}^{1}) \nleq ((g_{i}^{1})^{c},1)$, as, $g_{i}^{1} \nleq (g_{i}^{1})^{c}$ is true in a Boolean algebra. Now, let $(g,g^{\prime}) \in \mathcal{J}_{(\textbf{2}^{I})^{[2]}}$ be such that $(g,g^{\prime})  \nleq \sim(0,g_{i}^{1}) = ((g_{i}^{1})^{c},1)$. But we have $g^{\prime} \leq 1$, so 
for $(g,g^{\prime}) \nleq \sim(0,g_{i}^{1})$ to hold, we must have $g \nleq (g_{i}^{1})^{c}$.  $g_{i}^{1}$ is an atom of $\textbf{2}^{I}$ and $g \nleq (g_{i}^{1})^{c}$ imply $g_{i}^{1} \leq g$. Hence $(g_{i}^{1},g_{i}^{1}) \leq (g,g^{\prime})$, and we get
$(0,g_{i}^{1})^{*} = (g_{i}^{1},g_{i}^{1})$. Similarly, we have $(g_{i}^{1},g_{i}^{1})^{*} = (0,g_{i}^{1})$. Let us summarize these observations in the following lemma.
\begin{lemma} \label{lemma2.2}
The completely distributive De Morgan algebra $\textbf{3}^{I}$ has the following properties.  For each $i \in I$, {\rm $f_{i}^{a}, f_{i}^{1} \in \mathcal{J}_{\textbf{3}^{I}}$ and 
$(0,g_{i}^{1}), (g_{i}^{1},g_{i}^{1}) \in \mathcal{J}_{(\textbf{2}^{I})^{[2]}}$} we have,
\begin{enumerate}
      \item $(f_{i}^{a})^{*} = f_{i}^{1}$, $(0,g_{i}^{1})^{*} = (g_{i}^{1},g_{i}^{1})$.
			\item $(f_{i}^{1})^{*} = f_{i}^{a}$, $(g_{i}^{1},g_{i}^{1})^{*} = (0,g_{i}^{1})$.
\end{enumerate}
\end{lemma}
Now we return to Theorem  \ref{thm2.3}.\vskip 2pt
\noindent \textbf{Proof of Theorem \ref{thm2.3}}:\\ 
\noindent Let the Kleene algebra $\textbf{3}^{I}$ be given. Consider $\textbf{2}^{I}$ as a Boolean subalgebra of $\textbf{3}^{I}$.
Using the definition of $\phi$ (cf. Theorem \ref{join irr}) and its extension (cf. Lemma \ref{lemma2.1}), and using Lemma \ref{lemma2.2} we have, for each $i \in I$,
$$\phi((f_{i}^{a})^{*}) = \phi(f_{i}^{1}) = (g_{i}^{1},g_{i}^{1}) = \phi(f_{i}^{a})^{*},\\
\phi((f_{i}^{1})^{*}) = \phi(f_{i}^{a}) = (0,g_{i}^{1}) = \phi(f_{i}^{1})^{*}.$$
By Theorem \ref{thm2.8}, $\phi$ is an order isomorphism between $\mathcal{J}_{\textbf{3}^{I}}$ and $\mathcal{J}_{(\textbf{2}^{I})^{[2]}}$. Hence using Theorem \ref{thm2.10}, $\Phi$ is an isomorphism between the De Morgan algebras
$\textbf{3}^{I}$ and $(\textbf{2}^{I})^{[2]}$. As both the algebras are also Kleene algebras,
which are also equational algebras defined over De Morgan algebras, the De Morgan isomorphism $\Phi$ extends to Kleene isomorphism. 
\qed

\vskip 2pt 
Let us illustrate the above theorem through examples.
\begin{example}\label{eg3}
Consider the Kleene algebra $\textbf{3} := \{0,a,1\}$. Then $\mathcal{J}_{\textbf{3}} = \{a,1\}$.
For $\textbf{2} := \{0,1\}$,  $\textbf{2}^{[2]} = \{(0,0),(0,1),(1,1)\}$ and $\mathcal{J}_{\textbf{2}^{[2]}} = \{(0,1),(1,1)\}$.  Further, $a^{*} = 1$, $1^{*} = a$ and $(0,1)^{*} = (1,1)$ and $(1,1)^{*} = (0,1)$.\\
Then $\phi: \mathcal{J}_{\textbf{3}} \rightarrow \mathcal{J}_{\textbf{2}^{[2]}}$ is defined  as\\
\hspace*{3 cm} $\phi(a) := (0,1)$,\\
\hspace*{3 cm} $\phi(1): = (1,1)$.\\
Hence the extension map $\Phi: \textbf{3} \rightarrow \textbf{2}^{[2]}$ is given as\\
\hspace*{3 cm} $\Phi(a) := (0,1)$,\\
\hspace*{3 cm} $\Phi(1) := (1,1)$,\\
\hspace*{3 cm} $\Phi(0) := (0,0)$.\\
The diagrammatic illustration of this example is given in Figure \ref{fig3}. 
 
\begin{figure}[h] 
\begin{tikzpicture}[scale=.75]
                            inner sep=0pt]
    \draw [dotted] (0,1) -- (0,1);
	
    \draw (3,0) -- (3,2);
		\draw (3,2) -- (3,4);
		\draw [fill] (3,0) circle [radius = 0.1];
		\node [below] at (3,0) {0 = $\sim$ 1};
		\draw [fill] (3,2) circle [radius = 0.1];
		\node [right] at (3,2) {$a = \sim a$};
		\draw [fill] (3,4) circle [radius = 0.1];
		\node [above] at (3,4) {1 = $\sim$ 0};
		\node [right] at (1,2) {$\textbf{3}$ := };

		\node [right] at (6,2) {$\textbf{2}$ := };
		\draw (8,1) -- (8,3);
		\draw [fill] (8,1) circle [radius = 0.1];
		\node [below] at (8,1) {0 = $\sim$ 1};
		\draw [fill] (8,3) circle [radius = 0.1];
		\node [below] at (8,3) {1 = $\sim$ 0};
		
		\node [right] at (10,2) {$\textbf{2}^{[2]}$ := };
		\draw (13,0) -- (13,2);
		\node [below] at (13,0) {(0,0) = $\sim$ (1,1)};
    \draw [fill] (13,0) circle [radius = 0.1];
		\node [right] at (13,2) {(0,1) = $\sim$ (0,1)};
    \draw [fill] (13,2) circle [radius = 0.1];
		\draw (13,2) -- (13,4);
		\draw [fill] (13,4) circle [radius = 0.1];
		\node [above] at (13,4) {(1,1) = $\sim$ (0,0)};

\end{tikzpicture}
\caption{$\textbf{3} \cong \textbf{2}^{[2]}$}

\label{fig3}
\end{figure}

\end{example}
\begin{example}
Now, let us consider the Kleene algebra $\textbf{3}\times \textbf{3}$.\\
$\textbf{3}\times \textbf{3}: = \{(0,0),(0,a),(0,1),(a,1),(1,1),(a,0),(1,0),(1,a),(a,a)\}$.\\
$\mathcal{J}_{\textbf{3}\times \textbf{3}} = \{(0,a),(0,1),(a,0),(1,0)\}$ and \\
$(0,a)^{*} = (0,1), (0,1)^{*} = (0,a), (a,0)^{*} = (1,0), (1,0)^{*} = (a,0)$.\\ 
Consider the Boolean sub algebra $\textbf{2}\times \textbf{2} = \{(0,0),(0,1),(1,0),(1,1)\}$ of $\textbf{3}\times \textbf{3}$. For  convenience,  let us change the notations. We represent the set $\textbf{2}\times \textbf{2}$  and its elements as $\textbf{2}^{2} = \{0,x,y,1\}$, where   $(0,0)$ is replaced by  $0$, $(0,1)$ is replaced by $x$, 
$(1,0)$  is replaced by $y$, and $(1,1)$  is replaced by $1$. Then \\
$(\textbf{2}^{2})^{[2]} = \{(0,0),(0,x),(0,1),(0,y),(x,x),(x,1),(y,1),(y,y),(1,1)\}$, and\\ 
 $\mathcal{J}_{(\textbf{2}^{2})^{[2]}} = \{(0,x),(0,y),(x,x),(y,y)\}$.\\
Further, $(0,x)^{*} = (x,x), (x,x)^{*} = (0,x)$ and $(0,y)^{*} = (y,y), (y,y)^{*} = (0,y)$.\\
The diagrammatic illustration of the isomorphism between $\textbf{3}\times \textbf{3}$ and  
$(\textbf{2}^{2})^{[2]}$ is given in Figure \ref{fig4}.

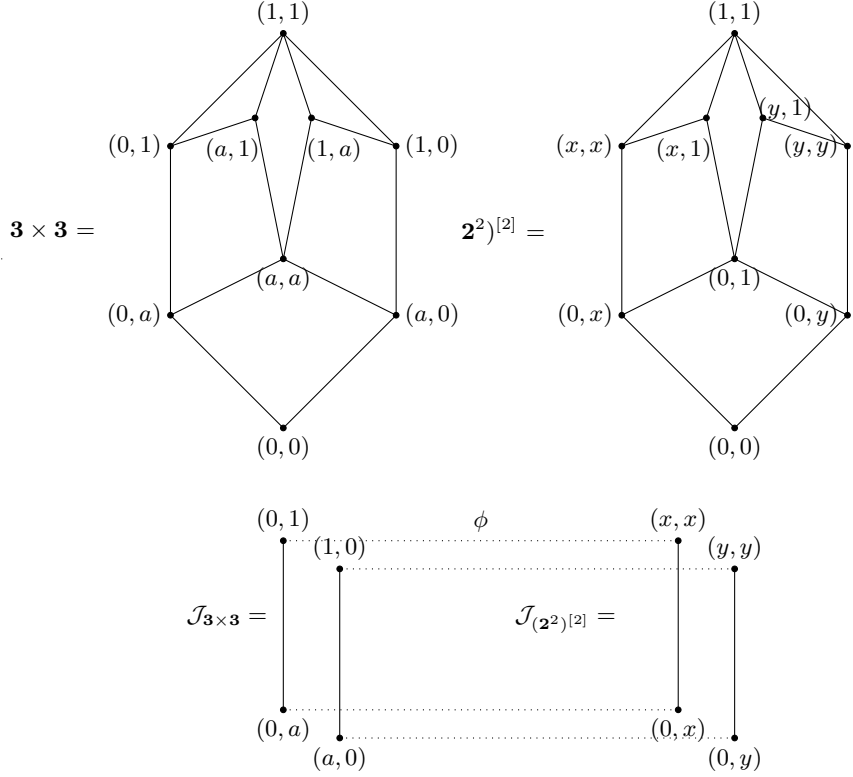
\begin{figure}[h] 

\begin{tikzpicture}[scale=.75]
                            inner sep=0pt]
    \draw [dotted] (0,1) -- (0,1);
	
		\node [right] at (0,1.5) {$\textbf{3}\times \textbf{3}$ = };
    \draw (3,3) -- (5,5);
		\draw (3,0) -- (3,3);
		\draw (5,-2) -- (3,0);
		\draw (7,0) -- (5,-2);
		\draw (5,5) -- (7,3);
		\draw (7,3) -- (7,0);
		\draw (5,1) -- (3,0);
		\draw (7,0) -- (5,1);
		\draw (3,3) -- (4.5,3.5);
		\draw (7,3) -- (5.5,3.5);
		\draw (4.5,3.5) -- (5,5);
		\draw (5.5,3.5) -- (5,5);
		\draw (4.5,3.5) -- (5,1);
		\draw (5.5,3.5) -- (5,1);
		
		\draw [fill] (3,3) circle [radius = 0.05];
		\node [left] at (3,3) {$(0,1)$};
	  \draw [fill] (5,5) circle [radius = 0.05];
		\node [above] at (5,5) {$(1,1)$};
		\draw [fill] (3,0) circle [radius = 0.05];
		\node [left] at (3,0) {$(0,a)$};
		\draw [fill] (5,-2) circle [radius = 0.05];
		\node [below] at (5,-2) {$(0,0)$};
		\draw [fill] (7,0) circle [radius = 0.05];
		\node [right] at (7,0) {$(a,0)$};
		\draw [fill] (7,3) circle [radius = 0.05];
		\node [right] at (7,3) {$(1,0)$};
		\draw [fill] (5,1) circle [radius = 0.05];
		\node [below] at (5,1) {$(a,a)$};
		\draw [fill] (4.5,3.5) circle [radius = 0.05];
		\node [below] at (4.1,3.3) {$(a,1)$};
		\draw [fill] (5.5,3.5) circle [radius = 0.05];
		\node [below] at (5.9,3.3) {$(1,a)$};
		
		
		\node [right] at (8,1.5) {$\textbf{2}^{2})^{[2]}$ = };
    \draw (11,3) -- (13,5);
		\draw (11,0) -- (11,3);
		\draw (13,-2) -- (11,0);
		\draw (15,0) -- (13,-2);
		\draw (13,5) -- (15,3);
		\draw (15,3) -- (15,0);
		\draw (13,1) -- (11,0);
		\draw (15,0) -- (13,1);
		\draw (11,3) -- (12.5,3.5);
		\draw (15,3) -- (13.5,3.5);
		\draw (12.5,3.5) -- (13,5);
		\draw (13.5,3.5) -- (13,5);
		\draw (12.5,3.5) -- (13,1);
		\draw (13.5,3.5) -- (13,1);
		
		\draw [fill] (11,3) circle [radius = 0.05];
		\node [left] at (11,3) {$(x,x)$};
	  \draw [fill] (13,5) circle [radius = 0.05];
		\node [above] at (13,5) {$(1,1)$};
		\draw [fill] (11,0) circle [radius = 0.05];
		\node [left] at (11,0) {$(0,x)$};
		\draw [fill] (13,-2) circle [radius = 0.05];
		\node [below] at (13,-2) {$(0,0)$};
		\draw [fill] (15,0) circle [radius = 0.05];
		\node [left] at (15,0) {$(0,y)$};
		\draw [fill] (15,3) circle [radius = 0.05];
		\node [left] at (15,3) {$(y,y)$};
		\draw [fill] (13,1) circle [radius = 0.05];
		\node [below] at (13,1) {$(0,1)$};
		\draw [fill] (12.5,3.5) circle [radius = 0.05];
		\node [below] at (12.1,3.3) {$(x,1)$};
		\draw [fill] (13.5,3.5) circle [radius = 0.05];
		\node [above] at (13.9,3.3) {$(y,1)$};
		
		
		\draw (5,-4) -- (5,-7);
		\draw (6, -4.5) -- (6,-7.5);
		\node [below] at (4,-5) {$\mathcal{J}_{\textbf{3}\times \textbf{3}} = $};
		\draw [fill] (5,-4) circle [radius = 0.05];
		\draw [fill] (5,-7) circle [radius = 0.05];
		\node [below] at (5,-7) {$(0,a)$};
		\node [above] at (5,-4) {$(0,1)$};
		\draw [fill] (6,-4.5) circle [radius = 0.05];
		\draw [fill] (6,-7.5) circle [radius = 0.05];
		\node [above] at (6,-4.5) {$(1,0)$};
		\node [below] at (6,-7.5) {$(a,0)$};
		
		\draw (12,-4) -- (12,-7);
		\draw (13, -4.5) -- (13,-7.5);
		\node [below] at (10,-5) {$\mathcal{J}_{(\textbf{2}^{2})^{[2]}} = $};
		\draw [fill] (12,-4) circle [radius = 0.05];
		\draw [fill] (12,-7) circle [radius = 0.05];
		\node [below] at (12,-7) {$(0,x)$};
		\node [above] at (12,-4) {$(x,x)$};
		\draw [fill] (13,-4.5) circle [radius = 0.05];
		\draw [fill] (13,-7.5) circle [radius = 0.05];
		\node [above] at (13,-4.5) {$(y,y)$};
		\node [below] at (13,-7.5) {$(0,y)$};
		
		\draw [dotted] (5,-7) -- (12,-7);
		\draw [dotted] (5,-4) -- (12,-4);
		\node [above] at (8.5,-4) {$\phi$};
		\draw [dotted] (6,-7.5) -- (13,-7.5);
		\draw [dotted] (13,-4.5) -- (6,-4.5);
		

\end{tikzpicture}
\caption{$\textbf{3}\times \textbf{3} \cong (\textbf{2}^{2})^{[2]}$}

\label{fig4}
\end{figure}

\end{example}

\section{The  logic  $\mathcal{L}_{K}$ for Kleene algebras and a 3-valued semantics}\label{section3} 

As mentioned earlier, Moisil in 1941 (cf. \cite{Cignoli07}) proved that $B^{[2]}$ forms a 3-valued LM algebra. Varlet 
(cf. \cite{Boicescu91}) noted the equivalence between  regular double Stone algebras and 3-valued LM algebras, whence  $B^{[2]}$ can   be given the structure of a regular double Stone algebra as well.   
So, while discussing the  logic corresponding to the structures $B^{[2]}$, one is naturally led to  3-valued {\L}ukasiewicz logic. 
Here, due to Proposition \ref{prop1} and Theorem \ref {thm2.1},  we focus on  $B^{[2]}$   as a {\it Kleene algebra}, and study the (propositional) logic corresponding to the class of Kleene algebras and the structures $B^{[2]}$. We denote this   system   as  $\mathcal{L}_{K}$, and present it below. 

Our approach to the study is 
motivated by Dunn's 4-valued semantics of the De Morgan consequence system \cite{Dunn99}. The 4-valued semantics arises from the fact that each element of a De Morgan algebra can be looked upon as a pair of sets. Now, using Stone's representation, 
each Boolean algebra is embeddable in a power set algebra, whence   $B^{[2]}$, for any Boolean algebra $B$, is embeddable  in
$\mathcal{P}(U)^{[2]}$, for some set $U$. Thus, because of Theorem \ref{thm2.1}, one can  say that each  element of a Kleene algebra can also be looked upon as a pair of sets.
As shown in Example \ref{eg3} above,  the Kleene algebra $\textbf{3} \cong \textbf{2}^{[2]}$. We exploit the fact that  $\textbf{3}$, in particular, can be represented as a Kleene algebra of pairs of sets, to get completeness of the logic $\mathcal{L}_{K}$ for Kleene algebras with respect to a 3-valued semantics. 

The {\it Kleene} axiom 
$\alpha~ \wedge \sim \alpha \vdash \beta~ \vee \sim \beta$, given by Kalman \cite{KAL58}, was studied by Dunn \cite{Dunn99,Dunn00} in the context of providing a 3-valued semantics for a fragment of relevance logic. He showed that the De Morgan consequence system coupled with the Kleene axiom (the resulting  consequence relation being denoted as $\vdash_{Kalman}$), is sound and complete with respect to  a semantic consequence relation (denoted $\models_{0,1}^{\textbf{3}_{R}}$) defined on 
$\textbf{3}_{R}$,  the {\it right hand chain} of the De Morgan lattice $\textbf{4}$ given in Figure \ref{fig6}. $\textbf{3}_{R}$ is the side of  $\textbf{4}$ in which the elements  are interpreted as $t$(rue), $f$(alse) and $b$(oth), and  $\models_{0,1}^{\textbf{3}_{R}}$ essentially incorporates truth {\it and} falsity preservation by valuations in its definition. He called this consequence system, the {\it Kalman consequence system}.
%
\begin{figure}[h]
\begin{tikzpicture}[scale=.75]
                            inner sep=0pt]
    \draw [dotted] (0,1) -- (0,1);

		\draw (7,4) -- (5,2);
		\draw [fill] (5,2) circle [radius = 0.1];
		\node [left] at (4.75,2) {$n = \sim n$};
		\draw (7,0) -- (5,2);
		\draw (7,4) -- (9,2);
		\draw [fill] (9,2) circle [radius = 0.1];
		\node [right] at (9.25,2) {$b= \sim b$};
		\draw (7,0) -- (9,2);
		\draw [fill] (7,0) circle [radius = 0.1];
		\node [below] at (7,-0.25) {$f = \sim t$};
		
		\draw [fill] (7,4) circle [radius = 0.1];
		\node [above] at (7,4.25) {$t = \sim f$};

\end{tikzpicture}
\caption{De Morgan lattice $\textbf{4}$}
\label{fig6}
\end{figure}
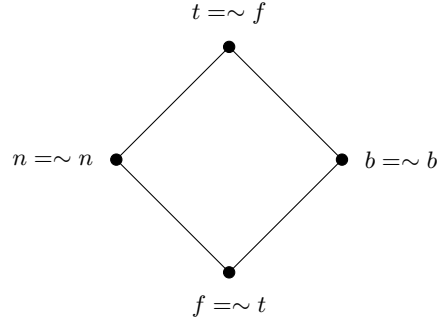
The completeness result for the Kalman consequence system  is obtained considering all 4-valued valuations restricted to $\textbf{3}_{R}$: the proof makes explicit reference to valuations on $\textbf{4}$. 

The logic $\mathcal{L}_{K}$ ($K$ for Kalman and Kleene) that we are considering in our work, has a consequence system that is $\vdash_{Kalman}$, with slight modifications. 
$\mathcal{L}_{K}$  is  shown to be sound and complete with respect to a 3-valued semantics that is based on the same idea underlying the consequence relation $\models_{0,1}^{\textbf{3}_{R}}$, viz. that of truth as well as falsity preservation. However, the definitions and proofs in this case, {\it do not refer to} $\textbf{4}$. 

\vskip 3pt
Let us present  $\mathcal{L}_{K}$. The language consists of
\begin{itemize} 
\item propositional variables:  $p,q,r, \ldots$. 
\item propositional constants: $\top, \bot$.
\item logical connectives: $\vee, \wedge, \sim$.
\end{itemize}
\noindent The well-formed formulae of the logic are defined through the scheme: \vskip 2pt \hspace*{2 cm} $\top~ | ~\bot~ |~ p ~| ~\alpha \vee \beta~ |~ \alpha \wedge \beta~ |~ \sim \alpha$. \vskip 2pt
\begin{notation} Denote the set of  
 well-formed formulae by $\mathcal{F}$. \end{notation}

The consequence relation $\vdash_{\mathcal{L}_{K}}$ is now given through the following postulates and rules, taken from \cite{Dunn99} and \cite{Dunn05}. These define reflexivity and transitivity of $\vdash$, introduction, elimination principles and the distributive law for the connectives $\wedge$ and $\vee$, contraposition and double negation laws for the negation operator $\sim$, the Kleene property for $\sim$, and  some basic requirements from  the propositional constants $\top, \bot$. Let $\alpha,\beta, \gamma \in \mathcal{F}$. 

\begin{definition} \label{defLK}
 {\rm ($\mathcal{L}_{K}$- postulates)
\begin{enumerate}
\item $\alpha \vdash \alpha$ \label{Lk1} 
\item $\alpha \vdash \beta, \beta \vdash \gamma~ /~ \alpha \vdash \gamma $.
\item $\alpha \wedge \beta \vdash \alpha$, $\alpha \wedge \beta \vdash \beta$.
\item $\alpha \vdash \beta, \alpha \vdash \gamma~ /~ \alpha \vdash \beta \wedge \gamma$.
\item $\alpha \vdash \gamma, \beta \vdash \gamma ~/~ \alpha \vee \beta \vdash \gamma$.
\item $\alpha \vdash \alpha \vee \beta$, $\beta \vdash \alpha \vee \beta$.
\item $\alpha \wedge (\beta \vee \gamma) \vdash (\alpha \vee \beta) \wedge (\alpha \vee \gamma)$ (Distributivity).
\item $\alpha \vdash \beta~ /~ \sim \beta \vdash \sim \alpha$ (Contraposition).
\item  $\sim \alpha \wedge \sim \beta \vdash \sim(\alpha \vee \beta)$ ($\vee$-linearity).
\item $\alpha \vdash \top$ (Top).
\item  $\bot \vdash \alpha$ (Bottom).
\item $\top \vdash \sim \bot$ (Nor). \label{Lk11}
\item $\alpha \vdash \sim \sim \alpha$.\label{dn1}
\item $\sim \sim \alpha \vdash \alpha$. \label{dn2}
\item $\alpha~ \wedge \sim \alpha \vdash \beta~ \vee \sim \beta$ (Kalman/Kleene).\label{kl}
\end{enumerate}}
\end{definition}

\noindent  Let us now consider any Kleene algebra $ (K,\vee,\wedge,\sim,0,1)$. We first define valuations on $K$.
\begin{definition}
A map $v: \mathcal{F} \rightarrow K$ is called a {\rm valuation} on $K$, if it satisfies the following properties for any $\alpha,\beta \in \mathcal{F}$.
\begin{enumerate}
\item $v(\alpha \vee \beta) = v(\alpha) \vee v(\beta)$.
\item $v(\alpha \wedge \beta) = v(\alpha) \wedge v(\beta)$.
\item $v(\sim \alpha) = \sim v(\alpha)$.
\item $v(\bot) = 0$. 
\item $v(\top) = 1$.
\end{enumerate} 
\end{definition}
 A consequent $\alpha \vdash \beta$ is {\it valid in K under the valuation} $v$, if $v(\alpha) \leq v(\beta)$. If the consequent is valid under all valuations on $K$, then it  is {\it valid in K}.  Let $\mathcal{A}$ be a class of Kleene algebras. If the consequent $\alpha \vdash \beta$ is valid in each algebra of $\mathcal{A}$, then we say $\alpha \vdash \beta$ {\it is valid in} $\mathcal{A}$, and denote it as $\alpha \vDash_{\mathcal{A}} \beta$.\vskip 2pt

Let $\mathcal{A}_{K}$ denote the class of {\it all} Kleene algebras. We have, in the classical manner,
\begin{theorem}\label{thm9}
$\alpha \vdash_{\mathcal{L}_{K}} \beta$ if and only if $\alpha \vDash_{\mathcal{A}_{K}} \beta$, for any 
$\alpha,\beta \in \mathcal{F}$.

\end{theorem}
\noindent Let us now focus on  valuations on the Kleene algebra  $B^{[2]}$. Then for $\alpha \in \mathcal{F}$, $v(\alpha)$ is a pair of  the form $(a,b)$. Suppose for $\beta \in \mathcal F$,  $v(\beta): = (c,d)$. By definition, the consequent $\alpha \vdash \beta$ is valid in $B^{[2]}$ under  $v$, when $v(\alpha) \leq v(\beta)$, i.e., $(a,b) \leq (c,d) $, or $a \leq c$ and 
$b\leq d$. 

Let $\mathcal{A}_{KB^{[2]}}$ denote the class of Kleene algebras formed by the sets $B^{[2]}$, for  {\it all}  Boolean algebras $B$. Then we have
\begin{theorem}\label{thm10}
$\alpha \vDash_{\mathcal{A}_{K}} \beta$ if and only if $\alpha \vDash_{\mathcal{A}_{KB^{[2]}}} \beta$, for any 
$\alpha,\beta \in \mathcal{F}$.
\end{theorem} 
\begin{proof}
Let $\alpha \vDash_{\mathcal{A}_{KB^{[2]}}} \beta$. Consider any Kleene algebra  $ (K,\vee,\wedge,\sim,0,1)$, and let $v$ be a valuation on $K$. By Theorem \ref{thm2.1}, there exists a Boolean algebra $B$ such that $K$ is embedded in 
$B^{[2]}$. Let  $\phi$ denote the embedding. It is  a routine verification that $\phi\circ v$ is  a valuation on $B^{[2]}$. The other direction is trivial, as $\mathcal{A}_{KB^{[2]}}$ is a subclass of $\mathcal{A}_{K}$. \qed
\end{proof}

On the other hand, as observed earlier, the structure  $B^{[2]}$ is embeddable  in 
$\mathcal{P}(U)^{[2]}$ for some set $U$, utilizing Stone's representation. Hence if $v$ is a valuation on $B^{[2]}$, it can be be extended to a valuation on $\mathcal{P}(U)^{[2]}$. Let $\mathcal{A}_{K\mathcal{P}(U)^{[2]}}$ denote the class of Kleene algebras of the form 
$\mathcal{P}(U)^{[2]}$, for {\it all} sets $U$. So, we get from Theorem \ref{thm10} the following.
\begin{corollary}\label{cor1}
$\alpha \vDash_{\mathcal{A}_{K}} \beta$ if and only if 
$\alpha \vDash_{\mathcal{A}_{K\mathcal{P}(U)^{[2]}}} \beta$, for any 
$\alpha,\beta \in \mathcal{F}$.
\end{corollary} 
 Following  \cite{Dunn99}, we now consider semantic consequence relations defined by valuations  $v: \mathcal{F} \rightarrow \textbf{3}$ on the Kleene algebra $\textbf{3}$. Let us re-label the elements of $\textbf{3}$ as $f,~u,~t,$ giving the standard truth value connotations.  

\begin{definition}\label{tf} Let $\alpha,\beta \in \mathcal F$.\vskip 2pt
\noindent \hspace*{1 cm}  $\alpha \vDash_{t} \beta$ if and only  if, if  $v(\alpha) = t$ then  $v(\beta) = t$~~ {\rm (Truth preservation)}.\\
\hspace*{1 cm}   $\alpha \vDash_{f} \beta$ if and only if, if $v(\beta) = f$ then  $v(\alpha) = f$  {\rm (Falsity preservation)}.\\
\hspace*{1 cm} $\alpha \vDash_{t,f} \beta$ if and only if,  $\alpha \vDash_{t} \beta$ {\it and} $\alpha \vDash_{f} \beta$.
\end{definition} 

\noindent \noindent We  adopt 
$\vDash_{t,f}$ as the semantic consequence relation for the logic $\mathcal{L}_{K}$. Note that the consequence relation $\vDash_{t}$ is  the consequence relation used in \cite{Urquhart01} to interpret the strong Kleene logic.  In case of Dunn's 4-valued semantics,  the consequence relations $\vDash_{t} $, $\vDash_{f} $ and $\vDash_{t,f}$ are defined using valuations on $\textbf{4}$. As shown in \cite{Dunn99}, all the three turn out to be equivalent. In order to capture the first-degree entailment fragment of relevance logic, Dunn subsequently uses the semantic consequence relation $\models_{0,1}^{\textbf{3}_{R}}$, defined by valuations  restricted to $\textbf{3}_{R}$, the  right  hand chain of $\textbf{4}$.  
Observe that for valuations on $ \textbf{3}$ that are being considered here, the consequence relations $\vDash_{t}$, $\vDash_{f}$ and $\vDash_{t,f}$ are not equivalent:
$\alpha~\wedge \sim \alpha \vDash_{t} \beta$, but $\alpha~ \wedge \sim \alpha \nvDash_{f} \beta$; $ \beta  \vDash_{f} \alpha ~\vee \sim \alpha$, but $ \beta  \nvDash_{t} \alpha~ \vee \sim \alpha$. 
\vskip 2pt

\begin{theorem} \label{thm16}
$\alpha \vDash_{\mathcal{A}_{K\mathcal{P}(U)^{[2]}}} \beta$ if and only if $\alpha \vDash_{t,f} \beta$, for any $\alpha,\beta \in \mathcal{F}$.
\end{theorem} 
\begin{proof}
Let $\alpha \vDash_{\mathcal{A}_{K\mathcal{P}(U)^{[2]}}} \beta$, and
 $v: \mathcal{F} \rightarrow \textbf{3}$ be a valuation. 
As we have already noted, $\textbf{3} \cong \mathcal{P}(U)^{[2]}$. If the correspondence is denoted by $\phi$,  $\phi \circ v$ is  a valuation on  $\mathcal{P}(U)^{[2]}$. Then $(\phi \circ v)(\alpha) \leq (\phi \circ v)(\beta)$ implies $v(\alpha) \leq  v(\beta)$. Thus if  $v(\alpha)=t$, we have $v(\beta)=t$, and if $v(\beta)=f$, then also $v(\alpha)=f$.\vskip 2pt

\noindent Now let $\alpha \vDash_{t,f} \beta$. Let $U$ be a set, and $\mathcal{P}(U)^{[2]}$ be the corresponding Kleene algebra.
  Let $v$ be a valuation on $\mathcal{P}(U)^{[2]}$ -- we need to show $v(\alpha) \leq v(\beta)$. For any $\gamma \in \mathcal{F}$ with $v(\gamma) := (A,B)$ and for each $x \in U$, define a map $v_{x}: \mathcal{F} \rightarrow \textbf{3}$ as\\
\hspace*{2 cm}$v_{x}(\gamma)  :=
\left\{
	\begin{array}{lll}
		t & \mbox{if }   x \in A \\
		u & \mbox{if } x \in B  \setminus A  \\
		f & \mbox{if } x \notin B .
	\end{array}
\right.$ \vskip 2pt 
\noindent We show that $v_{x}$ is  a valuation. \\ Consider any $\gamma, \delta \in \mathcal{F}$, with  $v(\gamma) := (A,B)$ and $v(\delta) := (C,D)$.
\begin{enumerate}
\item $v_{x}(\gamma \wedge \delta) = v_{x}(\gamma) \wedge v_{x}(\delta)$. \\Note that $v(\gamma \wedge \delta)= (A \cap C, B \cap D)$.\\
\noindent \un{Case 1} $v_{x}(\gamma) = t$ and $v_{x}(\delta) = t$: Then 
$x \in A \cap C$, and we have $v_{x}(\gamma \wedge \delta) = t = v_{x}(\gamma) \wedge v_{x}(\delta)$.\\
\noindent \un{Case 2} $v_{x}(\gamma) = t$ and $v_{x}(\delta) = u$:  $x \in A$, $x \in D$ and $x \notin C$, which imply $x \notin A \cap C$ but $x \in B \cap D$. Hence $v_{x}(\gamma \wedge \delta) = u = v_{x}(\gamma) \wedge v_{x}(\delta)$.\\
\noindent \un{Case 3}  $v_{x}(\gamma) = t$ and $v_{x}(\delta) = f$:  $x \in A$, $x \notin D$, which imply $x \notin B \cap D$. Hence $v_{x}(\gamma \wedge \delta) = f = v_{x}(\gamma) \wedge v_{x}(\delta)$.\\
\noindent \un{Case 4}  $v_{x}(\gamma) = u$ and $v_{x}(\delta) = f$: $x \notin A$ but $x \in B$ and $x \notin D$, which imply $x \notin B \cap D$. Hence $v_{x}(\gamma \wedge \delta) = f = v_{x}(\gamma) \wedge v_{x}(\delta)$.\\
\noindent \un{Case 5} $v_{x}(\gamma) = u$, $v_{x}(\delta) = u$:   $x \in B$ but $x \notin A$ and $x \in D$ but $x \notin C$.
So, $x \in B\cap D$ and $x \notin A\cap C$. Hence $v_{x}(\gamma \wedge \delta) = u = v_{x}(\gamma) \wedge v_{x}(\delta)$.\\
\noindent \un{Case 6}  $v_{x}(\gamma) = f$, $v_{x}(\delta) = f$:   $x \notin B$ and $x \notin D$. So, $x \notin B \cap D$.
Hence $v_{x}(\gamma \wedge \delta) = f = v_{x}(\gamma) \wedge v_{x}(\delta)$.

\item $v_{x}(\gamma \vee \delta) = v_{x}(\gamma) \vee v_{x}(\delta)$. \\Observe that $v(\gamma \vee \delta)= (A \cup C, B \cup D)$.\\
\noindent \un{Case 1}  $v_{x}(\gamma) = t$ and $v_{x}(\delta) = t$: Then $x \in A$, $x \in C$, which imply $x \in A \cup C$. Hence $v_{x}(\gamma \vee \delta) = t = v_{x}(\gamma) \vee v_{x}(\delta)$.\\
\noindent \un{Case 2}  $v_{x}(\gamma) = t$ and $v_{x}(\delta) = u$:  $x \in A$, $x \in D$ and $x \notin C$, in any way $x \in A \cup C$ . Hence $v_{x}(\gamma \vee \delta) = t = v_{x}(\gamma) \vee v_{x}(\delta)$.\\
\noindent \un{Case 3} $v_{x}(\gamma) = t$ and $v_{x}(\delta) = f$:  $x \in A$, $x \notin D$, which imply $x \in A \cup C$. Hence $v_{x}(\gamma \vee \delta) = t = v_{x}(\gamma) \vee v_{x}(\delta)$.\\
\noindent \un{Case 4} $v_{x}(\gamma) = u$ and $v_{x}(\delta) = f$:  $x \notin A$ but $x \in B$ and $x \notin D$, which imply $x \notin A \cup C$ but $x \in B \cup D$. Hence $v_{x}(\gamma \vee \delta) = u = v_{x}(\gamma) \vee v_{x}(\delta)$.\\
\noindent \un{Case 5}  $v_{x}(\gamma) = u$, $v_{x}(\delta) = u$:  $x \in B$ but $x \notin A$ and $x \in D$ but $x \notin C$.
So, $x \in B\cup D$ and $x \notin A\cup C$. Hence $v_{x}(\gamma \vee \delta) = u = v_{x}(\gamma) \vee v_{x}(\delta)$.\\
\noindent \un{Case 6} $v_{x}(\gamma) = f$, $v_{x}(\delta) = f$:  $x \notin B$ and $x \notin D$. So, $x \notin B \cup D$.
Hence $v_{x}(\gamma \wedge \delta) = f = v_{x}(\gamma) \wedge v_{x}(\delta)$.

\item$v_{x}(\sim\gamma) = \sim v_{x}(\gamma)$. \\Note that $ v(\sim \gamma)= (B^{c},A^{c})$.\\
\noindent \un{Case 1} $v_{x}(\gamma) = t$: Then $x \in A$, i.e. $x \notin A^{c}$. Hence 
          $v_{x}(\sim\gamma)  = f = \sim v_{x}(\gamma)$.\\
\noindent \un{Case 2} $v_{x}(\gamma) = u$:  $x \notin A$ but $x \in B$. So $x \in A^{c}$ and $x \notin B^{c}$. Hence 
          $v_{x}(\sim\gamma)  = u = \sim v_{x}(\gamma)$.\\	
\noindent \un{Case 3} $v_{x}(\gamma) = f$:  $x \notin B$, i.e. $x \in B^{c}$. So 
          $v_{x}(\sim\gamma)  = t = \sim v_{x}(\gamma)$.									

\end{enumerate}

\noindent Hence $v_{x}$ is a valuation in $\textbf{3}$. Now let us show that $v(\alpha) \leq v(\beta)$. Let  $v(\alpha) := (A^{\prime},B^{\prime})$, $v(\beta) := (C^{\prime},D^{\prime})$, and $x \in A^{\prime}$. Then $v_{x}(\alpha) = t$, and as $\alpha \vDash_{t,f} \beta$, by definition, $v_{x}(\beta) = t$. This implies $x \in C^{\prime}$, whence $A^{\prime} \subseteq C^{\prime}$.\\
On the other hand, if $x \notin D^{\prime}$, $v_{x}(\beta) = f$. Hence  $v_{x}(\alpha) = f$, so that $x \notin B^{\prime}$, giving  $B^{\prime} \subseteq D^{\prime}$.
\qed
\end{proof} 
\noindent Note that the above proof cannot be applied on the Kleene algebra $B^{[2]}$ instead of $\mathcal{P}(U)^{[2]}$, as we have used set representations explicitly.

\vskip 2pt An immediate consequence of Theorem \ref{thm9}, Corollary \ref{cor1} and Theorem  \ref{thm16} is the following.
\begin{theorem}\label{thm17}
$\alpha \vdash_{\mathcal{L}_{K}} \beta$ if and only if $\alpha \vDash_{t,f} \beta$, for any $\alpha,\beta \in \mathcal{F}$.
\end{theorem}


%

\section{Rough set semantics for ${\mathcal{L}_{K}}$} \label{section4}

Rough set theory, introduced by Pawlak \cite{pawlak82} in 1982, deals with a domain $U$ (set of objects) and an equivalence ({\it indiscerniblity}) relation $R$
on $U$. 
The pair $(U,R)$ is called an (Pawlak) {\it approximation space}.  
For any $A \subseteq U$, one defines the {\it lower} and {\it upper approximations} of $A$ in the approximation space $(U,R)$, denoted ${\sf L}A$ and ${\sf U}A$ respectively,  as follows.\\
\hspace*{2 cm} ${\sf L}A = \bigcup \{[x]: [x]\subseteq X\}$,\\
\hspace*{2 cm} ${\sf U}A = \bigcup \{[x]: [x] \cap X \neq \emptyset\}$. \hspace{1 cm} \hfill{(*)}\\
As the information about the objects of the domain is available modulo the equivalence classes in $U$, description of any concept, represented extensionally as the subset $A$ of $U$, is inexact, and one `approximates' the description from within and outside, through the lower and upper approximations respectively.
\noindent Unions of equivalence classes are termed as {\it definable} sets, signifying exact description in the context of the given information. In particular, sets of the form 
${\sf L}A$, ${\sf U}A$ are definable sets. 
\begin{definition}
Let $(U,R)$ be an approximation space. For each $A \subseteq U$, the ordered pair $({\sf L}A,{\sf U}A)$ is called a {\rm rough set} in $(U,R)$.
\begin{notation}  $\mathcal{RS} := \{({\sf L}A,{\sf U}A): A \subseteq U\}$.
\end{notation}
The ordered pair $(D_{1},D_{2})$, where $D_{1} \subseteq D_{2}$ and $D_{1},D_{2}$ are definable sets, is called a {\rm generalized rough set} in $(U,R)$.
 \begin{notation}  $\mathcal{D} $ denotes the collection of definable sets and $\mathcal{R} $  that of the generalized rough sets in $(U,R)$.\end{notation}
\end{definition}

In the following, we proceed to establish part (ii) of Theorem \ref{mainthm} (cf. Section \ref{section1}). In Section \ref{section3-v}, we formalize  the connection of rough sets with the 3-valued semantics being considered in this work. We end the section with a rough set semantics  for ${\mathcal{L}_{K}}$ (cf. Theorem \ref{thm18A}), obtained as a consequence of the representation results of Section \ref{sec4.1} below.

\subsection{Rough set  representation of Kleene algebras}\label{sec4.1}

Algebraically, the collection $\mathcal{D} $ of definable sets forms a complete atomic Boolean algebra in which atoms are the equivalence classes. The collection $\mathcal{RS}$ forms a distributive lattice -- in fact, it forms a Kleene algebra.   On the other hand, observe that $\mathcal{R}$ is the set $\mathcal{D}^{[2]} $ and hence forms a Kleene algebra (cf. Proposition \ref{prop1}) as well. $\mathcal{R}$ has earlier been studied, for instance, by Banerjee and Chakraborty in \cite{BC96}, and shown to form  {\it topological quasi-Boolean}, {\it pre-rough} and {\it rough} algebras. Note that, for an approximation space $(U,R)$, as sets $\mathcal{R}$ and $\mathcal{RS}$ may not be the same. So, it is natural to ask how $\mathcal{R}$ and $\mathcal{RS}$ differ as algebraic structures. The following  result mentioned in \cite{BC96} gives a connection between the two. The proof is not given in \cite{BC96}; we sketch it here, as  it is used  in the sequel. 
\begin{theorem} \label{thm3.11}
For any approximation space $(U,R)$, there exists an approximation space $(U^{\prime},R^{\prime})$ such that $\mathcal{R}$ corresponding to $(U,R)$ is order isomorphic to $\mathcal{R}^{\prime}$ corresponding to
$(U^{\prime},R^{\prime})$. Further, $\mathcal{R}^{\prime}$ = $\mathcal{RS}^{\prime}$, the latter denoting the collection of rough sets in the approximation space $(U^{\prime},R^{\prime})$ .
\end{theorem}
\begin{proof}
Let $(U,R)$ be the given approximation space. Consider the set $\textbf{A} := \{a\in U: |R(a)| = 1\}$, where $R(a)$ denotes the equivalence class of $a$ in $U$. So $\textbf{A}$ is the collection of all elements which are $R$-related only to themselves. Now, let us construct a set $\textbf{A}^{\prime}$ which consists of `dummy' elements not in $U$, indexed by the set $\textbf{A}$, i.e., $\textbf{A}^{\prime} := \{a^{\prime}: a \in \textbf{A}\}$. Let $U^{\prime} = U \cup \textbf{A}^{\prime}$.
Define an equivalence relation $R^{\prime}$ on $U^{\prime}$ as follows.\\
\hspace*{1 cm} If $a \in U$ then $R^{\prime}(a) := R(a) \cup \{x^{\prime}\in \textbf{A}^{\prime}: x \in R(a)\cap \textbf{A}\}$. \\
\hspace*{1 cm} If $a^{\prime} \in \textbf{A}^{\prime}$ then $R^{\prime}(a^{\prime}) := R(a)~(=\{a,a^{\prime}\})$.\\
\noindent Note that the number of equivalence classes in both the approximation spaces is the same.
Define the map $\phi: \mathcal{R} \rightarrow \mathcal{R}^{\prime}$ as $\phi(D_{1},D_{2}) := (D_{1}^{\prime},D_{2}^{\prime})$, where $D_{1}^{\prime} := D_{1} \cup \{x^{\prime}\in \textbf{A}^{\prime}: x \in D_{1}\cap \textbf{A}\}$ and $D_{2}^{\prime} := D_{2} \cup \{x^{\prime}\in \textbf{A}^{\prime}: x \in D_{2}\cap \textbf{A}\}$. Then $\phi$ is an order isomorphism.\qed
\end{proof}

Since $\mathcal{R}$ and $\mathcal{RS}$ for any approximation space $(U,R)$ form Kleene algebras, Theorem \ref{thm3.11} can easily be extended to Kleene algebras as follows.
\begin{theorem} \label{thm12}
Let $(U,R)$ be an approximation space. 
There exists an approximation space $(U^{\prime},R^{\prime})$ such that $\mathcal{R}$  corresponding to $(U,R)$ is Kleene isomorphic to $\mathcal{RS}^{\prime}~(= \mathcal{R}^{\prime})$ corresponding to
$(U^{\prime}, R^{\prime})$. 
\end{theorem}
\begin{proof}
Consider $(U^{\prime},R^{\prime})$ and $\phi$ as in Theorem \ref{thm3.11}.  $\phi$ is a lattice isomorphism, as the  restriction of $\phi$ to the completely join irreducible elements of the lattices $\mathcal{D}^{[2]} $ and $\mathcal{D}^{\prime [2]} $ is an order isomorphism (using Proposition \ref{prop3} and Lemma \ref{lemma2.1}). Let us now show that $\phi(\sim(D_{1},D_{2})) = \sim( \phi(D_{1},D_{2}))$.
To avoid confusion, we follow these notations: for  $X \subseteq U$ we use $X^{c_{1}}$ for the complement in $U$ and $X^{c_{2}}$ for the complement in $U^{\prime}$.

\noindent Now, $\phi(\sim(D_{1},D_{2}))  = \phi(D_{2}^{c_{1}},D_{1}^{c_{1}}) = ((D_{2}^{c_{1}})^{\prime},(D_{1}^{c_{1}})^{\prime})$.  By definition of $\phi$, we have:\\
\hspace*{2 cm} $(D_{2}^{c_{1}})^{\prime} = D_{2}^{c_{1}} \cup \{x^{\prime} \in \textbf{A}^{\prime}: x \in D_{2}^{c_{1}} \cap \textbf{A}\}$.\\
\hspace*{2 cm} $(D_{1}^{c_{1}})^{\prime} = D_{1}^{c_{1}} \cup \{x^{\prime} \in \textbf{A}^{\prime}: x \in D_{1}^{c_{1}} \cap \textbf{A}\}$.
\begin{claim}
 $(D_{2}^{c_{1}})^{\prime} = (D_{2}^{\prime})^{c_{2}}$, and $(D_{1}^{c_{1}})^{\prime} = (D_{1}^{\prime})^{c_{2}}$.
\end{claim}
{\it Proof of Claim}:
Let us first prove that\\
$(D_{2}^{c_{1}})^{\prime} = D_{2}^{c_{1}} \cup \{x^{\prime} \in \textbf{A}^{\prime}: 
x \in D_{2}^{c_{1}} \cap \textbf{A}\} = (D_{2}^{\prime})^{c_{2}} = 
(D_{2} \cup \{x^{\prime}: x \in D_{2}\cap \textbf{A}\})^{c_{2}} = 
(D_{2})^{c_{2}} \cap (\{x^{\prime} \in \textbf{A}^{\prime}: x\in D_{2}\cap \textbf{A}\})^{c_{2}}$. 

\noindent Let $X := \{x^{\prime} \in \textbf{A}^{\prime}: 
x \in D_{2}^{c_{1}} \cap \textbf{A}\}$ and $Y := \{x^{\prime} \in \textbf{A}^{\prime}: x\in D_{2}\cap \textbf{A}\}$.

\noindent Consider $a \in (D_{2}^{c_{1}})^{\prime} = D_{2}^{c_{1}} \cup X$.\\
\un{Case 1} $a \in D_{2}^{c_{1}}$:\\
\noindent As, $D_{2} \subseteq U$, $D_{2}^{c_{1}} \subseteq D_{2}^{c_{2}}$. Hence $a \in D_{2}^{c_{2}}$.
As $D_{2}^{c_{1}} \subseteq U$, $a \notin \textbf{A}^{\prime}$, whence 
$a \in Y^{c_{2}}$.
So $a \in (D_{2}^{\prime})^{c_{2}}$.\\
\un{Case 2} $a \in X$:\\
 $a = x^{\prime}$, where $x \in D_{2}^{c_{1}} \cap \textbf{A}$. As, $x^{\prime} \in R^{\prime}(x)$ and $D_{2}^{c_{2}}$ is the union of equivalence classes, in particular it contains $R^{\prime}(x)$. So $a = x^{\prime} \in D_{2}^{c_{2}}$.\\
 $x \in D_{2}^{c_{1}}$ implies $x \notin D_{2}$. Hence $a = x^{\prime} \in Y^{c_{2}}$.
So $a \in (D_{2}^{\prime})^{c_{2}}$.\\
Conversely,  let $a \in (D_{2}^{\prime})^{c_{2}} = 
(D_{2})^{c_{2}} \cap Y^{c_{2}}$.   \\
\un{Case 1}  $a \in U$:\\
$a \in D_{2}^{c_{2}} \Rightarrow a \in D_{2}^{c_{1}}$. Hence $a \in (D_{2}^{c_{1}})^{\prime}$.\\
\un{Case 2} $a \in \textbf{A}^{\prime}$:\\
$a \in Y^{c_{2}}$ implies
$a \in \{x^{\prime} \in \textbf{A}^{\prime}: x \in D_{2}^{c_{1}} \cap \textbf{A}\}$.  Hence $a \in (D_{2}^{c_{1}})^{\prime}$.


\noindent Similar arguments as above  show that $(D_{1}^{c_{1}})^{\prime} = (D_{1}^{\prime})^{c_{2}}$.\vskip 3pt
\noindent {\it Proof of Theorem \ref{thm12}}: \\$\phi(\sim(D_{1},D_{2}))  = \phi(D_{2}^{c_{1}},D_{1}^{c_{1}}) = ((D_{2}^{c_{1}})^{\prime},(D_{1}^{c_{1}})^{\prime}) = ((D_{2}^{\prime})^{c_{2}},(D_{1}^{\prime})^{c_{2}}) = \sim \phi(D_{1},D_{2})$.

\noindent Hence $\phi$ is a Kleene isomorphism.
\qed
\end{proof}
\noindent It is now not hard to see the correspondence between a complete atomic Boolean algebra and rough sets in an approximation space. 

\begin{theorem}\label{thm13}
Let $B$ be a complete atomic Boolean algebra. 
\blr
\item There exists an approximation space $(U,R)$ such that
      \bla
			\item $B \cong \mathcal{D}$.
			\item $B^{[2]}$ is Kleene isomorphic to $\mathcal{R}$.
			\ela
\item There exists an approximation space $(U^{\prime},R^{\prime})$ such that $B^{[2]} $ is 
      Kleene isomorphic to $\mathcal{RS}^{\prime}$.
\elr
\end{theorem}
\begin{proof}
Let $U $ denote the collection of all atoms of B, and $R$  the identity relation on $U$.  $(U,R)$ is the required approximation space.\qed
\end{proof}

Thus Theorem \ref{thm2.1} can be rephrased in terms of rough sets, and we get  Theorem \ref{mainthm}(ii).
\begin{corollary} \label{cor2}
Given a Kleene algebra $\mathcal{K}$, there exists an approximation space $(U,R)$ such that 
$\mathcal{K}$ can be embedded into $\mathcal{RS}$. In other words, every Kleene algebra is isomorphic to an algebra of rough sets in a Pawlak  approximation space.
\end{corollary}


\subsection{Rough sets and the Kleene algebra  $\textbf{3}$} \label{section3-v} 

\noindent The definitions  $(*)$ of lower and upper approximations of a set $A$ in an approximation space $(U,R)$ (cf. beginning of Section \ref{section4}), immediately yield the following interpretations.
\begin{enumerate}
\item \label{t} $x$  {\it certainly} belongs to $A$, if $x \in {\sf L}A$,  i.e. all  objects which are indiscernible to $x$
are in $A$.

\item \label{f} $x$  {\it certainly does not} belong to $A$, if $x \notin {\sf U}A$,  i.e. all objects which are indiscernible to $x$
are not in  $A$.

\item \label{p} Belongingness of $x$ to $A$ is {\it not certain, but possible}, 
if $x \in {\sf U}A$ but $x \notin {\sf L}A$. In rough set terminology, this is the case when $x$ is in the {\it boundary} of $A$: some objects indiscernible
to $x$ are in $A$, while some others, also indiscernible to $x$, are in $A^{c}$.
\end{enumerate}
\noindent  These interpretations have led to much work in the study of connections between 3-valued algebras or logics and rough sets, see for instance  \cite{Banerjee97,Iturr99,Pa98,Avron08,Ciucci14,Dunt97}. In particular, in \cite{Avron08}, Avron and Konikowska have 
 obtained a non-deterministic logical matrix and studied the 3-valued logic generated by this matrix, 
which is a sort of First order logic. 
In the direction 
of propositional logic, Banerjee and Chakraborty in \cite{BC96,Banerjee97} obtained {\it pre-rough logic} for the class of pre-rough algebras. It was subsequently proved by Banerjee in \cite{Banerjee97} that 3-valued {\L}ukasiewicz logic and pre-rough logic are equivalent, thereby imparting a rough set semantics to the former.

Let us  spell out the natural connections of the Kleene algebra $\textbf{3}$ with rough sets.

\vskip 2pt Observe that  $\textbf{3}$,  being isomorphic to $\textbf{2}^{[2]}$ (as noted earlier),  can also be viewed as a collection of rough sets in an approximation space, due to Theorem \ref{thm13}(ii). 

\begin{proposition} 
There exists an approximation space $(U,R)$ such that $\textbf{3} \cong \mathcal{RS}$.
\end{proposition}
\begin{proof}
Let $U := \{x,y\}$ and consider the equivalence relation  $R:=U \times U$ on $U$. The correspondence is depicted in  Figure \ref{fig5}.
\begin{figure}[h]
\begin{tikzpicture}[scale=.75]
    \draw [dotted] (0,1) -- (0,1);

    \draw (4,0) -- (4,2);
		\draw (4,2) -- (4,4);
		\draw [fill] (4,0) circle [radius = 0.1];
		\node [below] at (4,0) {$f = \sim t$};  
		\draw [fill] (4,2) circle [radius = 0.1];
		\node [right] at (4,2) {$u = \sim u $};
		\draw [fill] (4,4) circle [radius = 0.1];
		\node [above] at (4,4) {$t = \sim f $};
		\node [right] at (2,2) {$\textbf{3}$ := };
		
		\node [right] at (8,2) {$\cong$};
		
		\node [right] at (10,2) {$\mathcal{RS}$:=} ;
		
		\draw (12,0) -- (12,2);
		\draw (12,2) -- (12,4);
		\draw [fill] (12,0) circle [radius = 0.1];
		\node [below] at (12,0) {$(\emptyset,\emptyset) = ({\sf L}\emptyset,{\sf U}\emptyset)$};
			\node [below] at (13,-0.5) {$ = \sim ({\sf L}U,{\sf U}U)$};  
		\draw [fill] (12,2) circle [radius = 0.1];
		\node [right] at (12,2) {$(\emptyset,U) = ({\sf L}x,{\sf U}x) $};
		\node [below] at (14,2) {$= \sim({\sf L}x,{\sf U}x) $};
		\draw [fill] (12,4) circle [radius = 0.1];
		\node [above] at (12,4) {$(U,U) = ({\sf L}U,{\sf U}U) = \sim({\sf L}\emptyset,{\sf U}\emptyset)$};

\end{tikzpicture}
\caption{$\textbf{3} \cong \mathcal{RS}$}\qed
\label{fig5}
\end{figure}
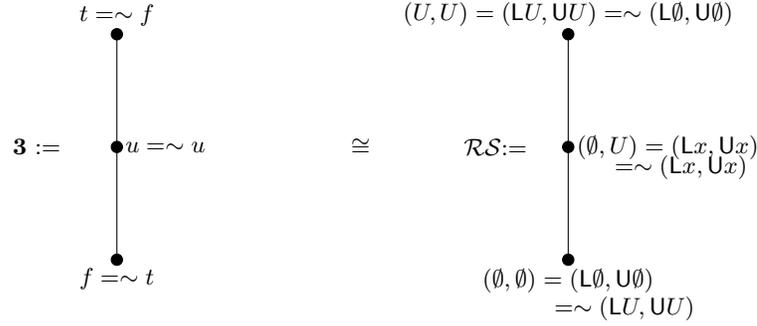
\end{proof}

\noindent On the other hand, interpretations 1-3 above  give rise to a correspondence  with the set $\textbf{3}:= \{f,u,t\}$, that assigns to every  $x \in U$ and rough set  $({\sf L}A,{\sf U}A)$ in $(U,R)$,  the value $t$ when   $x \in {\sf L}A$, $u$ when $x \in {\sf U}A \setminus {\sf L}A$, and $f$ in case  $x \notin {\sf U}A$. 
As one can see, this is akin to the valuation defined  in the proof of Theorem \ref{thm16}. In fact, using results of the previous sections, we  can  formally link the 3-valued semantics being considered here,  and rough sets. \vskip 2pt

Let $\mathcal{A}_{K\mathcal{RS}}$ denote the class containing the collections   $\mathcal{RS}$ of rough sets over {\it all} possible approximation spaces $(U,R)$. 

\begin{theorem} For any 
$\alpha,\beta \in \mathcal{F}$, \vskip 2pt
\blr \item $\alpha \vDash_{\mathcal{A}_{K}} \beta$ if and only if $\alpha \vDash_{\mathcal{A}_{K\mathcal{RS}}} \beta$, \vskip 2pt
\item $\alpha \vDash_{\mathcal{A}_{K\mathcal{RS}}} \beta$ if and only if  $\alpha \vDash_{t,f} \beta$.
\elr
\end{theorem}

\vskip 2pt In the process, we have   obtained a rough set semantics for   $\mathcal{L}_{K}$.
\begin{theorem}\label{thm18A} For any 
$\alpha,\beta \in \mathcal{F}$, $\alpha \vdash_{\mathcal{L}_{K}} \beta$ if and only if $\alpha \vDash_{\mathcal{A}_{K\mathcal{RS}}} \beta$.
\end{theorem}


\section{Perp semantics for the logic $\mathcal{L}_{K}$} \label{section5}

Dunn's framework of perp semantics for negations is of logical, philosophical as well as algebraic importance. On the  one hand, it provides relational 
semantics for various logics with negations (cf. e.g., \cite{Dunn99,Dunn05}), interpreting the negations as `impossibility' or `unnecessity' operators. On the other hand, one can give representations of various algebras as set algebras \cite{Dunn94}.
In this section, we characterize the Kleene consequent $\alpha~ \wedge \sim \alpha \vdash \beta~ \vee \sim \beta$ in  Dunn's framework of negations. Further, one obtains a representation of Kleene algebras through duality. \\First, we briefly present  the basics of perp semantics. For details, one may refer to \cite{Dunn94,Dunn00,Dunn99,Dunn05}. 
\begin{definition}
A {\rm compatibility frame} is a triple $(W,C,\leq)$ with the following properties.
\begin{enumerate}
\item $(W,\leq)$ is a partially ordered set.
\item $C$ is a binary relation on $W$ such that for $x,y,x^{\prime},y^{\prime} \in W,$ if $x^{\prime} \leq x$, $y^{\prime} \leq y$ and $xCy$ then $x^{\prime}Cy^{\prime}$.
\end{enumerate}
$C$ is called a {\rm compatibility relation} on $W$.\vskip 2pt
\noindent A {\rm perp frame} is a tuple $(W,\perp,\leq)$, where $\perp$, the {\rm perp relation} on $W$, is the complement of the compatibility relation $C$.
\end{definition}
\noindent Recall the  syntax of the logic $\mathcal{L}_{K}$ as defined  in Section \ref{section3}. 
\begin{definition}
A relation $\vDash$ between points of $W$ and propositional variables in $\mathcal{P}$  is called an {\rm evaluation}, if it satisfies the {\rm hereditary} condition:
\begin{itemize}
\item if $x \vDash p$ and $x \leq y$ then $y \vDash p$. 
\end{itemize}
\end{definition}

\noindent Recursively, an evaluation can be extended to $\mathcal{F}$ as follows.
\begin{enumerate}
\item $x \vDash \alpha \wedge \beta$ iff $x \vDash \alpha$ and $x \vDash \beta$.
\item $x \vDash \alpha \vee \beta$ iff $x \vDash \alpha$ or $x \vDash \beta$.
\item $x \vDash \top$.
\item $x \nvDash \bot$.
\item $x \vDash \sim \alpha$ iff $\forall y(xCy \Rightarrow y \nvDash \alpha)$.
\end{enumerate}
\noindent Then one can easily show that $\vDash$ satisfies the hereditary condition for  all formulae in  $\mathcal{F}$. 
For the compatibility frame $\textbf{F} := (W,C,\leq)$, the pair $(\textbf{F}, \vDash)$ for an evaluation $\vDash$ is called a {\it model}. The notion of validity is introduced next in the following  (usual) manner.\\
A consequent $\alpha \vdash \beta$ is {\it valid in a model} $(\textbf{F}, \vDash)$, denoted as $\alpha \vDash_{(\textbf{F}, \vDash)} \beta$,  if and only if, if $x \vDash \alpha$ 
      then $x \vDash \beta$, for each $x \in W$.\\
$\alpha \vdash \beta$ is {\it valid in the compatibility frame}  $\textbf{F}$, denoted as $\alpha \vDash_{\textbf{F}} \beta$,  if and only if  $\alpha \vDash_{(\textbf{F}, \vDash)} \beta$ 
for every model  $(\textbf{F},\vDash)$. \\
 Let $\mathbb{F}$ denote  a class of compatibility frames. $\alpha \vdash \beta$ is {\it valid in} $\mathbb{F}$, denoted as $\alpha \vDash_{\mathbb{F}} \beta$,  if and only if   $\alpha \vDash_{\textbf{F}} \beta$ for  every frame belonging to  $\mathbb{F}$. \vskip 2pt 

 Following \cite{Dunn05}, let $K_{i}$ denote the logic whose  postulates are \ref{Lk1} to \ref{Lk11} of the logic $\mathcal{L}_{K}$ (cf. Definition \ref{defLK}). In \cite{Dunn05} it has been proved that 
 $K_{i}$ is the minimal logic which is sound and complete with respect to the class of all compatibility frames. Frame completeness results for various normal  logics with negation have been proved using the canonical frames for the  logics. Let us we give the definitions for the canonical frame.
\begin{definition}
A set of sentences $P$ in a  logic $\Lambda$ is called a {\rm prime theory} if
\begin{enumerate}
\item $\alpha \vdash \beta$ holds and $\alpha \in P$, then $\beta \in P$.
\item $\alpha, \beta \in P$ then $\alpha \wedge \beta \in P$.
\item $\top \in P$ and $\bot \notin P$.
\item $\alpha \vee \beta \in P$ implies $\alpha \in P$ or $\beta \in P$.
\end{enumerate}
\end{definition}
Let $W_{c}$ be the collection of all non-trivial prime theories of $\Lambda$. Define a relation $C_{c}$ on $W_{c}$ as $P_{1}C_{c}P_{2}$ if and only if, for all sentences $\a$ of $\Lambda$, 
$\sim \alpha \in P_{1}$
implies $\alpha \notin P_{2}$. The tuple $(W_{c},C_{c},\subseteq)$ is the {\it canonical frame} for $\Lambda$.\\
A logic $\Lambda$ is called {\it canonical}, if its canonical frame is a frame for $\Lambda$.\vskip 2pt

 Let us now consider the canonical frame for the logic $\mathcal{L}_{K}$.
Note that $\mathcal{L}_{K}$ contains $K_{i}$ along with the postulates (\ref{dn1}) $\alpha \vdash \sim\sim\alpha$, (\ref{dn2}) $\sim\sim\alpha \vdash \alpha$ and (\ref{kl}) $\alpha ~\wedge \sim \alpha \vdash \beta~ \vee \sim \beta$ of Definition \ref{defLK}. The consequents $\alpha \vdash \sim\sim\alpha$
and $\sim\sim\alpha \vdash \alpha$ have been characterized by Dunn (for e.g.,\cite{Dunn05}) and Restall \cite{Restall00} respectively as follows.
\begin{theorem}\label{De Morgan}\noindent
\begin{enumerate}
\item $\alpha \vdash \sim\sim \alpha$ is valid in the class of all compatibility frames satisfying the following frame condition:\\
\hspace*{2 cm} $\forall x \forall y(xCy \rightarrow yCx)$.\vskip 2pt 
\item $\sim\sim\alpha \vdash  \alpha$ is valid in the class of all compatibility frames satisfying the frame condition:\\
\hspace*{2 cm} $\forall x \exists y(xCy \wedge \forall z(yCz \rightarrow  z \leq x))$.\vskip 2pt 
\item  The canonical frame  for $\mathcal{L}_{K}$ satisfies both the above frame conditions.
\end{enumerate}
\end{theorem}
It remains for us to characterize the Kleene consequent $\alpha~ \wedge \sim \alpha \vdash \beta~ \vee \sim \beta$ with an appropriate frame condition, and prove that the canonical frame for $\mathcal{L}_{K}$ satisfies the  condition.
\begin{theorem}\label{Kleene}
$\alpha~ \wedge \sim \alpha \vdash \beta~ \vee \sim \beta$ is valid in a compatibility frame, if and only if the compatibility 
relation satisfies the following first order property:\\
\hspace*{3 cm} $\forall x(xCx \vee \forall y(xCy \rightarrow y \leq x))$. \hspace{1 cm} \hfill{(*)} \\
\noindent The canonical frame  for $\mathcal{L}_{K}$  satisfies $(^*)$.
\end{theorem}
\begin{proof}
Consider any compatibility frame $(W,C, \leq)$, let $(^*)$ hold, and let $x \in W$.

\noindent Suppose $xCx$, then  $x \nvDash \alpha~ \wedge \sim \alpha$, and 
 trivially, if $x \vDash \alpha~ \wedge \sim \alpha$
then $x \vDash \beta~ \vee \sim \beta$.

 Now suppose  $\forall y(xCy \rightarrow y \leq x)$ is true.
Let $x \nvDash \beta$ and  $xCz$.
Then $z \leq x$, whence $z \nvDash \beta$. So, by definition $x \vDash \sim\beta$. Hence $x \vDash \beta~ \vee \sim\beta $. 
Hence in any case if 
 $x \vDash \alpha ~\wedge \sim \alpha$
then $x \vDash \beta~ \vee \sim \beta$. \\
\noindent Let $(^*)$ not hold.\\
This implies $\exists x(not(x Cx) \wedge \exists y(xCy \wedge y\nleq x))$. Let us define, for any $z,w \in W$,\\
\hspace*{2 cm} $z \vDash p$ if and only if  $x \leq z$ and $\sim xCz$, \\
\hspace*{2 cm} $w \vDash q$ if and only if $y \leq w$.\\
We show that $\vDash$ is an evaluation. Let $z \vDash p$ and 
$z \leq z^{\prime}$. Then $x \leq z^{\prime}$. If $xCz^{\prime}$, then by the frame condition on $C$ we have\\
\hspace*{2 cm} $x \leq x, z \leq z^{\prime}$ $xCz^{\prime}$ imply $xCz$,\\
which is a contradiction to the fact that $z \vDash p$. 

Furthermore, $x \vDash p$, as $x \leq x$ and $not(xCx) $. We also have $x \vDash \sim p$: if $xCw$ for any $w \in W$ then by definition, $w \nvDash p$. Hence, $x \vDash \sim p$ and so $x \vDash p~ \wedge \sim p$.\\
On the other hand, $x \nvDash q$ as $y \nleq x$. By the assumption, 
$xCy$ and $y \vDash q$, hence $x \nvDash \sim q$. So, we have $x \vDash p ~\wedge \sim p$ but
 $x \nvDash q ~\vee \sim q$.\vskip 2pt

\noindent $Canonicity:$\\ 
\noindent Let $\sim PC_{c}P$. By definition of $C_{c}$, then there exists an $\alpha \in \mathcal{F}$  such that $\alpha$, $\sim \alpha$ $\in P$, but then $ \alpha~ \wedge \sim \alpha \in P$. Hence for all $\beta \in \mathcal{F}$, $\beta~ \vee \sim \beta \in P$. So, for all $\beta$,  either $\beta \in P$ or $\sim \beta \in P$.\\
Now let $PC_{c}Q$ and let $\beta \in Q$. Hence $\sim \beta \notin P$ but as from above, $\sim PC_{c}P$,
we have $\beta \in P$. So, $Q \subseteq P$.  \qed
\end{proof}

\begin{definition} Let us call a compatibility frame $(W,C,\leq)$ a {\rm  Kleene frame} if it satisfies the following frame conditions.
\begin{enumerate}
			\item $\forall x \forall y(xCy \rightarrow yCx)$.
			\item $\forall x \exists y(xCy \wedge \forall z(yCz \rightarrow  z \leq x))$.
			\item  $\forall x(xCx \vee \forall y(xCy \rightarrow y \leq x))$.
\end{enumerate}
\noindent Denote by $\mathbb{F}_{K}$, the class of all Kleene frames.
\end{definition}

\noindent  We have then arrived at
\begin{theorem}\label{thm20}
The following are all equivalent, for any $\a,\b \in \mathcal{F}$. 
\bla \item $\alpha \vdash_{\mathcal{L}_{K}} \beta$. 
\item $\alpha \vDash_{\mathcal{A}_{K}} \beta$. 
\item $\alpha \vDash_{\mathcal{A}_{K\mathcal{RS}}} \beta$. 
\item $\alpha \vDash_{t,f} \beta$. 
\item $\alpha \vDash_{\mathbb{F}_{K}} \beta$.
\ela
\end{theorem}

\section{Conclusions} \label{section6}

In case of   Boolean algebras and classical propositional logic, or De Morgan algebras and De Morgan logic,  the algebraic semantics and the 2 or 4-valued semantics (respectively) are equivalent -- due to representation theorems for the two classes of algebras. Here, analogously, we have the result for the class of Kleene algebras, that the algebraic semantics and a 3-valued semantics (given by $\vDash_{t,f}$) of the logic $\mathcal{L}_{K}$ of Kleene algebras are equivalent. This is due to the representation theorem (Theorem \ref{thm2.1}) one is able to prove: any Kleene algebra is embeddable in $B^{[2]}$, for some Boolean algebra $B$.
 
Moreover, we have proved that the 3-valued semantics of  $\mathcal{L}_{K}$ can translate into a rough set (Theorem \ref{thm18A}) as well as perp semantics (Theorem \ref{thm20}). It is thus established that $\mathcal{L}_{K}$ can be imparted equivalent semantics from different perspectives.

\end{document}